\documentclass{amsart}
\usepackage{amsmath} 
\usepackage{amssymb}
\usepackage{mathrsfs}

\newtheorem{theorem}{Theorem}[section] 
\newtheorem{claim}[theorem]{Claim}

\theoremstyle{definition}
\newtheorem{definition}[theorem]{Definition}
\newtheorem{dc}[theorem]{Definition/Claim}
\newtheorem{example}[theorem]{Example}

\newtheorem{observation}[theorem]{Observation} 

\newtheorem{question}[theorem]{Question}
\newtheorem{convention}[theorem]{Convention}
\newtheorem{conjecture}[theorem]{Conjecture}
\newtheorem{context}[theorem]{Context}
\newtheorem{discussion}[theorem]{Discussion}

\newtheorem{hypothesis}[theorem]{Hypothesis}

\theoremstyle{remark}
\newtheorem{remark}[theorem]{Remark}

\newtheorem{conclusion}[theorem]{Conclusion}

\newcommand{\rest}{{\restriction}}

\newcommand{\wilog}{{\rm without loss of generality}}

\newcommand{\then}{{\underline{then}}}
\newcommand{\when}{{\underline{when}}}

\newcommand{\mn}{{\medskip\noindent}}
\newcommand{\sn}{{\smallskip\noindent}}

\newcommand{\und}{\underline}

\newcommand{\cE}{{\mathcal E}}

\newcommand{\bbD}{{\mathbb D}}

\newcommand{\cF}{{\mathcal F}}

\newcommand{\cP}{{\mathcal P}}

\newcommand{\bbR}{{\mathbb R}}

\newcommand{\cS}{{\mathscr S}}
\newcommand{\cT}{{\mathcal T}}

\newcommand{\cW}{{\mathcal W}}

\newcommand{\cY}{{\mathcal Y}}

\newcount\skewfactor
\def\mathunderaccent#1#2 {\let\theaccent#1\skewfactor#2
\mathpalette\putaccentunder}
\def\putaccentunder#1#2{\oalign{$#1#2$\crcr\hidewidth
\vbox to.2ex{\hbox{$#1\skew\skewfactor\theaccent{}$}\vss}\hidewidth}}

\newenvironment{PROOF}[2][\proofname.]
   {\begin{proof}[#1]\renewcommand{\qedsymbol}{\textsquare$_{\rm #2}$}}
   {\end{proof}}

\begin{document}

\title {PCF arithmetic without and with choice \\
 Sh938}
\author {Saharon Shelah}
\address{Einstein Institute of Mathematics\\
Edmond J. Safra Campus, Givat Ram\\
The Hebrew University of Jerusalem\\
Jerusalem, 91904, Israel\\
 and \\
 Department of Mathematics\\
 Hill Center - Busch Campus \\ 
 Rutgers, The State University of New Jersey \\
 110 Frelinghuysen Road \\
 Piscataway, NJ 08854-8019 USA}
\email{shelah@math.huji.ac.il}
\urladdr{http://shelah.logic.at}
\thanks{The author would like to thank the Israel Science Foundation
 for partial support of this research. (Grant No. 242/03).  Publ. 938.
The author thanks Alice Leonhardt for the beautiful typing.
First Typed - 08/July/28}  


\date{February 24, 2010}

\begin{abstract}
We deal with relatives of GCH which are provable.  In
particular we deal with rank version of the revised GCH.  Our
motivation was to find such results when only weak versions of the
axiom of choice are assumed but some of the results gives us
additional information even in ZFC.  We also start to deal with pcf
for pseudo-cofinality (in ZFC with little choice).
\end{abstract}

\maketitle
\numberwithin{equation}{section}
\setcounter{section}{-1}

\section*{Annotated Content}

\noindent
\S0 \quad Introduction, pg.2
\begin{enumerate}
\item[${{}}$]   [We present introductory remarks mainly to \S3,\S4.]  
\end{enumerate}

\noindent
\S1 \quad Preliminaries, pg.3
\begin{enumerate}
\item[${{}}$]   [We present some basic definitions and claims, mostly 
used later.]
\end{enumerate}

\noindent
\S2 \quad Commuting ranks, pg.8
\begin{enumerate}
\item[${{}}$]   [If we have ideals $D_1,D_2$ on sets $Y_1,Y_2$ and
a $Y_1 \times Y_2$-rectangle $\bar\alpha$ of ordinals, we can compute rank
in two ways: one is first apply rk$_{D_1}$ on each row and then
rk$_{D_2}(-)$ on the resulting column.  In the other we first apply
rk$_{D_2}(-)$ on each column and then rk$_{D_1}(-)$ on the resulting
row.  We give sufficient conditions for an inequality.  We use (ZFC + DC
and) weak forms of choice like AC$_{Y_\ell}$ or AC$_{{\cP}(Y_\ell)}$.]
\end{enumerate}

\noindent
\S3 \quad Rank systems and a Relative of GCH, pg.13
\begin{enumerate}
\item[${{}}$]   [We give a framework to prove a relative of the main
theorem of \cite{Sh:460} dealing with ranks.  We do it with weak form
of choice (DC + AC$_{< \mu}),\mu$ a limit cardinal, this
give new information also in ZFC.]
\end{enumerate}

\noindent
\S4 \quad Finding systems, pg.21
\begin{enumerate}
\item[${{}}$]   [The main result in \S3 deals with an abstract
setting.  Here we find an example, a singular limit of measurables.
Note that even under ZFC this gives information on ranks.]
\end{enumerate}

\noindent
\S5 \quad Pseudo true cofinality, pseudo tcf, pg.23
\begin{enumerate}
\item[${{}}$]   [We look again at the pcf$(\bar \alpha)$, but only
for $\aleph_1$-complete filters using pseudo-cofinality and the
cofinalities not too small.  Under such restrictions we get parallel
to pcf basic results.]
\end{enumerate}

\section {Introduction} 

In \cite{Sh:460} and \cite{Sh:513}, \cite{Sh:829} we prove in ZFC = ZF
+ AC relatives of G.C.H.  Here mainly we are interested in  
relatives assuming only weak forms of choice, \underline{but} some results add
information even working in ZFC, in particular a 
generalization of \cite{Sh:460} for ranks.  
Always we can assume ZF + DC.

Our original motivation was
\begin{conjecture}
\label{y0.3}
Assume ZF + DC and $\mu$ a limit cardinal such that AC$_{< \mu}$
and $\mu$ is strong limit.  For every ordinal $\gamma$, for some
$\kappa < \mu$, for any $\alpha < \mu$ and $\kappa$-complete filter
$D$ on $\alpha$ we have rk$_D(\gamma) = \gamma$.

Here we get an approximation to it, i.e. for $\mu$ a limit of
measurables restricting ourselves to ultrafilters; this is conclusion
\ref{5e.17} deduced by applying Theorem \ref{m4.12} to Claim \ref{e5.13}.
Can we do it with $\mu = \beth_\omega$?

Also we would like to weaken AC$_{< \mu}$; this is interesting per
se and as then we will be able to combine \cite{Sh:835} +
\cite{Sh:513} - see below.  We intend to try in \cite{Sh:F955}; 
starting with $\bar J = \langle J_n:n < \omega\rangle$ such that
IND$(\bar J)$ or something similar.  
\end{conjecture}

It may be illuminating to compare the present result with (see \cite[V]{Sh:g}).
\begin{claim}
\label{y0.7}
If $\kappa \ge \theta > \aleph_0,\lambda \ge 2^{2^\kappa}$ 
\underline{then} the following conditions are equivalent:
\medskip

\noindent
\begin{enumerate}
\item[$(*)_1$]    for every $\theta$-complete filter $D$ on $\kappa$,
we have {\rm rk}$_D(\lambda^+) = \lambda^+$ 
\smallskip

\noindent
\item[$(*)_2$]   $\alpha < \lambda^+ \Rightarrow 
\text{\rm rk}_D(\alpha) < \lambda^+$ for every 
$\theta$-complete filter $D$ on $\kappa$
\smallskip

\noindent
\item[$(*)_3$]   there is no ${\cF} \subseteq {}^\kappa \lambda$
of cardinality $\ge \lambda^+$ and $\theta$-complete filter $D$ in
$\kappa$ such that $f_1 \ne f_2 \in {\cF} \Rightarrow f_1 \ne_D f_2$.
\end{enumerate}
\end{claim}
\bigskip

\noindent
Also we can in \ref{y0.7} replace $\lambda^+$ by a cardinal of cofinality $>
2^{2^\kappa}$.  So the result in \cite{Sh:460} implies a weak version
of the conjecture above, say on $|\text{rk}_D(\alpha)|$, but the present
one gives more precise information.  On the other hand, the present
conjecture is not proved for $\mu = \beth_\omega$, also it seems 
less accommodating to the possible results with $\aleph_\omega$
instead of $\beth_\omega$ in \cite{Sh:513} below
$2^{2^{\aleph_\omega}}$.

\begin{question}
\label{y0.9}
In \cite{Sh:908} can we prove that the rank is small?
\end{question}

\begin{discussion}  In \ref{y0.4} below we 
present examples showing some limitations.

Below part (1) of the example shows that Claim \ref{c3.3} cannot be improved
too much and part (2) shows that Conclusion \ref{5e.17} cannot be
improved too much.  In fact, in conjecture \ref{y0.3} if we demand only
``$\mu$ is a limit cardinal" then it consistently fails.  This implies
that we cannot improve too much other results in \S3,\S4.

We may wonder how to compare the result in \cite{Sh:460} and
Conjecture \ref{y0.3} even in ZFC.
\end{discussion}

\begin{example}
\label{y0.4}
1) If $D_\ell = \text{ dual}([\kappa_\ell]^{< \kappa_\ell})$ for 
$\ell=1,2$ (so if $\kappa_\ell$ is regular then $D_\ell = 
\text{ dual}(J^{\text{bd}}_{\kappa_\ell}))$ and $\kappa_2 < \kappa_1$ \then \,
   $D_2$ does not 2-commute with $D_1$, i.e. $\boxplus^2_{D_1,D_2}$
from Definition \ref{c3.1} fail.

\noindent
2) Consistently with ZFC, for every $n$,
   rk$_{J^{\text{bd}}_{\aleph_n}}(\aleph_\omega) > \aleph_\omega$.
\end{example}

\begin{PROOF}{\ref{y0.4}}
1) Let $A = \kappa_1$ and let $f_2 \in {}^{\kappa_2}\text{Ord}$ be
   constantly 1 hence by Definition \ref{z0.27} and Claim
   \ref{z0.29}(3) the ideal $J_2 = J[f_2,D_2]$ is $[\kappa_2]^{<
   \kappa_2}$.  Choose a function $h:\kappa_1 \rightarrow \kappa_2$ 
and $(\forall \beta < \kappa_2)(\exists^{\kappa_1} 
\alpha < \kappa_1)(h(\alpha) = \beta)$
   and let $\langle B_\alpha:\alpha \in A\rangle$ be such that we have
$B_\alpha := \kappa_2 \backslash h(\alpha)$.

So if $A_* \in D_1,B_* \in J^+_2$ then for some $\alpha_* < \kappa_1$
we have $A_* \supseteq \kappa_1 \backslash \alpha_*$ and $B_* \subseteq
\kappa_2,|B_*| = \kappa_2$ and choose $t \in B_*$ and then choose
$s \in A_*$ such that $h(s) = t+1$, 
such $s$ exists by the choice of $h$ so $(s,t) \in
A_* \times B_*$ but $(s,t) \notin \{s\} \times B_s$, contradiction.

\noindent
2) Assume that the sequence
$\langle 2^{\aleph_n}:n < \omega \rangle$ is increasing with
supremum $> \aleph_\omega$ and in
cf$({}^{(\aleph_n)}(\aleph_n),<_{J^{\text{bd}}_{\aleph_n}})$ there is an
increasing sequence of length $\aleph_{\omega +1}$ for each $n \in
[1,\omega)$ hence it follows that
rk$_{J^{\text{bd}}_{\aleph_n}}(\aleph_\omega) > 
\text{ rk}_{J^{\text{bd}}_{\aleph_n}}(\aleph_n) \ge \aleph_n$ for $n \in
[1,\omega)$.
\end{PROOF}

We may hope to prove interesting
things in ZF + DC by division to cases: if \cite{Sh:835} apply fine,
if not then we have a strict $\bold p$.  \underline{But} we need AC$_{< \mu}$
to prove even clause (f) in \ref{m4.3}, see \cite{Sh:F955}.  
We may consider that even in ZFC, probably \cite{Sh:908} 
indicate that we can use weaker assumptions.

Let us say something on our program on set theory with little choice
of which this work is a part.  We always ``know" that the axiom of
choice is true.  In addition we 
had thought that there is no interesting general
combinatorial set theory without AC (though equivalence of version of
choice, inner model theory and some other exist).  Concerning the
second point, since \cite{Sh:497} our opinion changed and 
have thought that there is an interesting such
set theory, with ``bounded choice" related to pcf.  More specifically
\cite{Sh:497} seems to prove that such theory is not empty.  Then
\cite{Sh:835} suggest to look at axioms of choice ``orthogonal" to
``$\bold V = \bold L[\bbR]$", e.g. demand then ${}^{\omega \ge}\alpha$
can be well ordered (and weaker relatives).  The results say that the
universe is somewhat similar to universes gotten by Easton like
forcing, blowing up $2^\lambda$ for every regular $\lambda$ without
well ordering the new $\cP(\lambda)$.  Continuing this Larson-Shelah
\cite{LrSh:925} generalize classical theorem on splitting a stationary
subset of a regular $\lambda$ consisting of ordinals of cofinality $\kappa$. 

In \cite{Sh:F955} we shall continue this work.  In particular, we continue
\S5 to get a parallel of the pcf theorem and more.
Recall that in \cite{Sh:513} in ZFC 
we get connections between the
existence of independent sets and a strong form of \cite{Sh:460}.  We
prove related theorems on rank.  

We thank the referee for many corrections and remarks.

\section{Preliminaries} 

\begin{context}
\label{z0.0}  
1) We work in ZF in all this paper.

\noindent
2) We try to say when we use DC but assuming it always makes no great
harm.

\noindent
3) We shall certainly mention the use of any additional form of
choice, mainly  AC$_A$.

\noindent
4) In \ref{z0.11} - \ref{z0.29} we quote definitions and claims to
be used, see \cite{Sh:835}.    The rest of \S1 is used only in \S5.
\end{context}

\begin{definition}
\label{z0.11} 
1) A filter $D$ on $Y$ is $(\le B)$-complete \underline{when}: 
if $\langle A_t:t \in B\rangle \in {}^B D$ then $A
:= \cap\{A_t:t \in B\} \in D$.  We can instead say ``$|B|^+$-complete" even if
$|B|^+$ is not well defined.

\noindent
1A) A filter $D$ on $Y$ is pseudo $(\le B)$-complete \underline{when} if
$\langle A_t:t \in B\rangle \in {}^B D$ then $\cap\{A_t:t \in B\}$ is
not empty (so adopt the conventions of part (1)).

\noindent
2) For an ideal $J$ on a set $Y$ let dual$(J) = \{Y \backslash X:X 
\in J\}$, the dual ideal and Dom$(J) = Y$, abusing notation we assume
$J$ determines $Y$.  

\noindent
3) For a filter $D$ on a set $Y$ let dual$(D) =
\{Y \backslash X:X \in D\}$, Dom$(D) = Y$.  We may use properties
defined for filter $D$ for the dual ideal (and vice versa).

\noindent
4) For a filter $D$ on $Y$ let $D^+ = \{A \subseteq Y:Y \backslash A
\notin D\}$ and for an ideal $J$ on $Y$ let $J^+ =
(\text{dual}(J))^+$.  
\end{definition}

\begin{remark}  It may be interesting to try to assume that relevant
filters are just pseudo $(\le B)$-complete instead of $(\le
B)$-complete.  Now \ref{z0.55} clarify the connection to some extent,
but presently we do not pursue this direction.
\end{remark}

\begin{definition}
\label{z0.13}   
$\bold C$ is the class of sets
$A$ such that AC$_A$, the axiom of choice for $A$ non-empty sets, holds.
\end{definition}

\begin{definition}
\label{z0.15}  
1) $\theta(A) = \text{ Min}\{\alpha$: there is no function 
from $A$ onto $\alpha\}$.

\noindent
2) $\Upsilon(A) = \text{ Min}\{\alpha$: there is no one-to-one
function from $\alpha$ into $A\}$ so $\Upsilon(A) \le \theta(A)$.
\end{definition}

\begin{definition}
\label{z0.21}  
1) For $D$ a filter on $Y$ and $f,g \in {}^Y\text{Ord}$ let $f <_D g$
or $f < g$ mod $D$ means that $\{s \in Y:f(s) < g(s)\} \in D$;
similarly for $\le,=,\ne$.

\noindent
2) For $D$ a filter on $Y$ and $f \in {}^Y\text{Ord}$ and $\alpha 
\in \text{ Ord } \cup \{\infty\}$ we define when rk$_D(f) = \alpha$ by
induction on $\alpha$:
\medskip

\noindent
\begin{enumerate}
\item[$\circledast$]   For $\alpha < \infty$, rk$_D(f) = \alpha$ iff
$\beta < \alpha \Rightarrow \text{ rk}_D(f) \ne \beta$ and for every $g \in
{}^Y\text{Ord}$ satisfying $g <_D f$ there is $\beta < \alpha$ such
that rk$_D(g) = \beta$.
\end{enumerate}
\medskip

\noindent
3) We can replace $D$ by the dual ideal.
\end{definition}

\begin{observation}
\label{z0.24} 
1) Let $D$ be a pseudo $\aleph_1$-complete filter on $Y$.  If
$f,g \in {}^Y \text{\rm Ord}$ and $f \le_D g$ then {\rm rk}$_D(f) \le 
\text{\rm rk}_D(g)$ and so if $f =_D g$ then {\rm rk}$_D(f) = 
\text{\rm rk}_D(g)$.

\noindent
2) If $D_\ell$ is a pseudo $\aleph_1$-complete filter on $Y$ for
$\ell=1,2$ \underline{then} $D_1 \subseteq D_2 \wedge f \in
{}^Y\text{\rm Ord} \Rightarrow \text{\rm rk}_{D_1}(f) 
\le \text{\rm rk}_{D_2}(f)$.
\end{observation}

\begin{proof}  Easy.
\end{proof}

\begin{claim}
\label{z0.23}  Assume $D$ is a filter on $Y$
such that $D$ is $\aleph_1$-complete or just pseudo
$\aleph_1$-complete (see Definition \ref{z0.11}(1A)).

\noindent
1)  {\rm [DC]} For $f \in {}^Y\text{\rm Ord}$, 
in \ref{z0.21}, {\rm rk}$_D(f)$ is always an ordinal, i.e. $< \infty$.

\noindent
2)  {\rm [DC]} If $\alpha \le \text{\rm rk}_D(f)$ \underline{then} 
for some $g \in \prod\limits_{t \in Y} (f(t)+1)$ we 
have $\alpha = \text{\rm rk}_D(g)$. If $\alpha < \text{\rm rk}_D(f)$ we can 
add $g <_D f$ and we can demand $(\forall y \in Y)(g(y) < f(g) \vee
g(y) = 0 = f(y))$.

\noindent
2A) If {\rm rk}$_D(f) < \infty$ then part (2) holds for $f$ (without
assuming DC).

\noindent
3) If $f,g \in {}^Y\text{\rm Ord}$ and $f <_D g$ and {\rm rk}$_D(f) 
< \infty$ \then \, {\rm rk}$_D(f) < \text{\rm rk}_D(g)$.

\noindent
4) For $f \in {}^Y\text{\rm Ord}$ we have {\rm rk}$_D(f) > 0$ iff $\{t \in
Y:f(t) > 0\} \in D$.

\noindent
5) If $f,g \in {}^Y\text{\rm Ord}$ and $f =g+1$ then {\rm rk}$_D(f) 
= \text{\rm rk}_D(g)+1$.
\end{claim}

\begin{PROOF}{\ref{z0.23}}  
Straight, e.g.

\noindent
2A) We prove this by induction on $\beta = \text{ rk}_D(f)$.  If $\beta \le
\alpha$ there is nothing to prove.  If $\beta = \alpha +1$ by the
definition, there is $g <_D f$ such that rk$_D(g) \ge \alpha$, now by
part (3) we have rk$_D(g) < \text{ rk}_D(f)$ which means rk$_D(g) <
\alpha +1$, so together rk$_D(g) = \alpha$ and let $g' \in {}^Y\text{Ord}$ be
defined by $g'(s)$ is $g(s)$ if $g(s) < f(s)$ and is 0 if $g(s) \ge
f(s)$ so $g' <_D f$ and $g' \le_D g \le_D g'$ hence rk$_D(g') = \text{
rk}_D(g) = \alpha$ is as required.  

Lastly, if $\beta > \alpha +1$ by the definition there is $f' <_D f$
such that rk$_D(f') \ge \alpha +1$ and by \ref{z0.24}(1) \wilog \,
$t \in Y \Rightarrow f'(t) \le f(t)$ and by part (3) rk$_D(f') <
\text{ rk}_D(f)$ so we can apply the induction hypothesis to $f'$.
\end{PROOF}

\begin{claim}
\label{z0.25} 
1) [{\rm AC}$_{\aleph_0}$]  If $D$ is an $\aleph_1$-complete
filter on $Y$ and $f \in {}^Y${\rm Ord} and $Y = 
\cup\{Y_n:n < \omega\}$ \then \, {\rm rk}$_D(f) = 
 \text{\rm min}\{\text{\rm rk}_{D+Y_n}(f):n < \omega$ and $Y_n \in D^+\}$.

\noindent
2) [{\rm AC}$_{\cW}$] If $D$ is a $|{\cW}|^+$-complete filter on
$Y,{\cW}$ infinite and $f \in {}^Y${\rm Ord} and 
$\cup\{Y_t:t \in {\cW}\} \in D$ \then \, {\rm rk}$_D(f) = 
{ \text{\rm min\/}}\{\text{\rm rk}_{D+Y_t}(f):t \in {\cW}$ and
$Y_t \in D^+\}$.
\end{claim}

\begin{PROOF}{\ref{z0.25}}
Like \cite{Sh:71}.

\noindent
1) By part (2).

\noindent
2) Note that by AC$_{\cW}$ necessarily $\{t:Y_t \in D^+\}$ is non-empty.
The inequality $\le$ is obvious (i.e. by \ref{z0.24}(2)).  We prove by
induction on the ordinal $\alpha$ that $(\forall v \in {\cW})
[Y_v \in D^+ \Rightarrow \text{ rk}_{D+Y_v}(f) \ge \alpha]
\Rightarrow \text{ rk}_D(f) \ge \alpha$.

For $\alpha = 0$ and $\alpha$ is limit this is trivial.  

For $\alpha = \beta +1$, we assume $(\forall v \in {\cW})
[Y_v \in D^+ \Rightarrow
\text{ rk}_{D+Y_v}(f) \ge \alpha > \beta]$ so by Definition
\ref{z0.21} it follows that $[v \in {\cW} \wedge Y_v \in D^+ 
\Rightarrow (\exists g)(g \in {}^Y\text{Ord } \wedge 
g <_{D+Y_v} f \wedge \text{ rk}_{D+Y_v}(g) \ge
\beta]$ hence, if $v \in \cW \wedge Y_v \in D^+$ then 
$\{t \in Y:f(t)=0\} = \emptyset$ mod $(D + Y_v)$,
i.e. $\{v:f(v) = 0\} \cap Y_v = \emptyset$ mod $D$.  As this holds for
every $v \in \cW$ and $D$ is $|\cW|^+$-complete clearly
 we have $\{t \in Y:f(t) = 0\} = \emptyset$ mod $D$.  
We can by \ref{z0.24}(1) replace $f$ by $f' \in {}^Y\text{Ord}$ when $\{v \in
Y:f(v) = f'(v)\} \in D$ so \wilog \, $t \in Y \Rightarrow f(t) > 0$.

But ${\cW} \in \bold C$, hence by \ref{z0.23}(2A) 
there is a sequence $\langle g_v:v
\in {\cW}_*\rangle$ such that ${\cW}_* := \{v \in {\cW}:Y_v
\in D^+\}$ and $g_v \in {}^Y\text{Ord},g_v <_{D+Y_v} f,
\text{rk}_{D+Y_v}(g_v) \ge \beta$ and $t \in Y \Rightarrow g_v(t) <
f(t)$ so $g_v < f$.

As $D$ is $|{\cW}|^+$-complete necessarily $Y_* :=
\cup\{Y_v:v \in {\cW} \backslash {\cW}_*\} = \emptyset$ mod $D$,
but $\cup\{Y_v:v \in {\cW}\} \in D$ hence $Y_* = \cup\{Y_v:v \in
{\cW}_*\}$ belongs to $D$.
Define $g \in {}^Y\text{Ord}$ by $g(s) = \text{ min}\{g_u(s):u \in
{\cW}_*$ satisfies $s \in Y_u\}$ if $s \in Y_*$ and 0 
if $s \in Y \backslash Y_*$.

Hence $(\cup\{Y_v:v \in {\cW}_*\}) \in D$ and
$g \in {}^Y\text{Ord}$ and $g <_D f$ (and even $g < f$) so by the induction
hypothesis
\medskip

\noindent
\begin{enumerate}
\item[$\odot$]    it suffices to prove $v \in {\cW}_* \Rightarrow
\text{ rk}_{D+Y_v}(g) \ge \beta$.
\end{enumerate}
\medskip

\noindent
Fix $v \in {\cW}_*$, and for each $u \in {\cW}_*$ let
$Y_{v,u} := \{t \in Y_u \cap Y_v:g(t)= g_u(t)\}$ so by the choice of
$g(t)$ we have
\medskip

\noindent
\begin{enumerate}
\item[$\boxplus_1$]   if $v \in \cY_*,t \in Y_y$ \then \, for
some $u \in \cW_*$ we have $t \in Y_{y,x} \subseteq Y_y$ and $g(t)= g_u(t)$.
\end{enumerate}

Hence
\medskip

\noindent
\begin{enumerate}
\item[$\boxplus_2$]  $\langle Y_{v,u}:u \in {\cW}_*\rangle$ 
exists and $\cup\{Y_{v,u}:u \in {\cW}_*\} = Y_v \in (D+Y_v)$.
\end{enumerate}

Now
\medskip

\noindent
\begin{enumerate}
\item[$\boxplus_3$]   if $u \in {\cW}_* \wedge
Y_{v,u} \in (D + Y_v)^+$ then rk$_{D+Y_{v,u}}(g) \ge \beta$.
\end{enumerate}
\medskip

\noindent
[Why?  By the choice of $Y_{v,u}$ we have 
$g = g_u$ mod$(D+Y_{v,u})$ hence rk$_{D+Y_{v,u}}(g) =
\text{ rk}_{D+Y_{v,u}}(g_u)$, also $Y_{v,u} \subseteq Y_u$ hence
$D+Y_{v,u} \supseteq D + Y_u$ which by \ref{z0.24}(2) implies
rk$_{D+Y_{v,u}}(g_u) \ge \text{ rk}_{D+Y_u}(g_u)$ which is $\ge
\beta$.  Together we are done.]

By $\boxplus_2 + \boxplus_3$ and the induction hypothesis it follows
that $v \in {\cW}_* \Rightarrow \text{ rk}_{D+y_v}(g) \ge \beta$ so
by $\odot$ we are done. 
\end{PROOF}

\begin{definition}
\label{z0.27}   For $Y,D,f$ in \ref{z0.21} 
let $J[f,D] =: \{Z \subseteq Y:Y \backslash Z \in D$ 
or $(Y \backslash Z) \in D^+ \wedge \text{ rk}_{D+Z}(f) > \text{ rk}_D(f)\}$. 
\end{definition}

\begin{claim}
\label{z0.29}   [\rm{DC+AC}$_{\cY}$]  Assume $D$ 
is an $\aleph_1$-complete $|{\cY}|^+$-complete filter on $Y$.

\noindent
1) If $f \in {}^Y{\text{\rm Ord\/}}$ \then \, $J[f,D]$ is an
$\aleph_1$-complete and $|{\cY}|^+$-complete ideal on $Y$. 

\noindent
2) If $f_1,f_2 \in {}^Y\text{Ord}$ and $J = J[f_1,D] = J[f_2,D]$
\und{then} {\rm rk}$_D(f_1) < \text{\rm rk}_D(f_2) 
\Rightarrow f_1 < f_2$ mod $J$ and {\rm rk}$_D(f_1) = \text{\rm
rk}_D(f_2) \Rightarrow f_1 = f_2$ mod $J$.

\noindent
3) If $f \in {}^Y\text{\rm Ord}$ is e.g. constantly 1 then $J[f,D] =
   \text{\rm dual}(D)$.

\noindent
4) If $f \in {}^Y\text{Ord}$ and $A \in (J[f,D])^+$ \then \, ($A \in
   D^+$ and) {\rm rk}$_{D+A}(f) = \text{\rm rk}_D(f)$.
\end{claim}

\begin{PROOF}{\ref{z0.29}}
1) By \ref{z0.25}.

\noindent
2) As $J$ is an ideal on $Y$ (by part (1)) this 
should be clear by the definitions; that is, let $A_0 := \{t \in
Y:f_1(t) < f_2(t)\},A_1 := \{t \in Y:f_1(t) = f_2(t)\}$ and $A_2 :=
\{t \in Y:f_1(t) > f_2(t)\}$.  Now $\langle A_0,A_1,A_0\rangle$ is a
partition of $Y$.  

First, assume $A_0 \in J^+$, then by the definition of
$J[f_1,D]$ we have $\neg(\text{rk}_D(f_1) < \text{
rk}_{D+A_0}(f_1))$;  i.e. rk$_{D+A_0}(f_1) \le \text{ rk}_D(f_1)$ and
so by \ref{z0.24}(2) we have rk$_D(f_1) = \text{ rk}_{D+A_0}(f_1)$.
Now as $A_0 \in J^+$, by the choice of $A_0,f_1 <_{D+A_0} f_2$ hence
rk$_D(f_1) = \text{ rk}_{D+A_0}(f_1) 
< \text{ rk}_{D+A_0}(f_2) = \text{ rk}_D(f_2)$.

\noindent
[Why?  By the previous sentence, by \ref{z0.23}(3), by the previous
sentence respectively.]

Second, similarly if $A_2 \in J^+$ then $f_2 < f_1$ mod $(D+A_2)$ and
rk$_D(f_1) > \text{ rk}_D(f_2)$.  

Lastly, if $A_1 \in J^+$ then by \ref{z0.24}(1)
$f_1 = f_2$ mod $(D+A_1)$ hence rk$_{D+A_1}(f_1) = \text{ rk}_{D+A_1}(f_2)$ and
rk$_D(f_1) = \text{ rk}_{D+A_1}(f_2) = \text{ rk}_{D+A_1}(f_2) = 
\text{ rk}_D(f_2)$.

By the last three paragraphs at most one of $A_0,A_1,A_2$ belongs to
$J^+$ and as $A_0 \cup A_1 \cup A_2 = Y$ at least one of $A_0,A_1,A_2$
belongs to $J^+$, so easily we are done.

\noindent
3)  Obvious. 

\noindent
4) Proved inside the proof of part (2).
\end{PROOF}

\begin{definition}
\label{z0.32}  
1) Let FIL$^{\text{cc}}_S(Y)$ or
FIL$^{\text{pcc}}_S(Y)$ be the set of $D$ such that:
\smallskip

\noindent
$D$ is a filter on the set $Y$ which is $|S|^+$-complete and is
$\aleph_1$-complete or is psuedo $|S|^+$-complete and psuedo 
$\aleph_1$-complete.

\noindent
2) Let FIL$_{\text{cc}}(Y)$ or FIL$_{\text{pcc}}(Y)$ be
FIL$^{\text{cc}}_\emptyset$ or  FIL$^{\text{pcc}}_\emptyset$.

\noindent
3) Omitting $Y$ means for some $Y$ and then we let $Y = \text{
   Dom}(D)$.  Without enough choice, the minimal $(\le S)$-complete
   filter extending a filter $D$ is gotten in stages.
\end{definition}

\begin{definition}
\label{z0.51}  
1) For a filter $D$ on $Y$ and set $S$ we define 
comp$_{S,\gamma}(D)$ by induction on $\gamma \in
\text{ Ord } \cup\{\infty\}$.
\medskip

\noindent
\und{$\gamma = 0$}:  comp$_{S,\gamma}(D) = D$
\medskip

\noindent
\und{$\gamma =$ limit}:  comp$_{S,\gamma}(D) =
\cup\{\text{comp}_{S,\beta}(D):\beta < \gamma\}$
\medskip

\noindent
\und{$\gamma = \beta +1$}:  comp$_{S,\gamma}(D) = \{A \subseteq Y:A$
belongs to comp$_{S,\beta}(D)$ or 
include the intersection of some $S$-sequence of members of 
comp$_{S,\beta}(D)$, i.e. $\cap\{A_s:s \in S\}$, where 
$\langle A_s:s \in S\rangle$ is a
sequence of members of comp$_{S,\beta}(D)\}$.

\noindent
2) Similarly for a family ${\cS}$ of sets replacing $S$ by ``some
member of ${\cS}$", e.g. we define com$_{\in {\cS},\gamma}(D)$ by induction 
on $\gamma$ using $(\in {\cS})$-sequences, i.e. $S$-sequence for
some $S \in {\cS}$.

\noindent
3) If $\gamma = \infty$ we may omit it.  We say that $D$ is a pseudo
   $(\le S,\gamma)$-complete 
when $\emptyset \notin \text{ comp}_{S,\gamma}(D)$.  
\end{definition}

\begin{observation}
\label{z0.55}   
1) If $D$ is a filter on $Y$ and $S$ is a set, \und{then}:
\medskip

\noindent
\begin{enumerate}
\item[$(a)$]   $\langle\text{\rm comp}_{S,\gamma}(D):\gamma \in
\text{\rm Ord } \cup \{\infty\}\rangle$ is an $\subseteq$-increasing 
sequence of filters of $Y$ (starting with $D$)
\smallskip

\noindent
\item[$(b)$]   if {\rm comp}$_{S,\gamma +1}(D) = 
\text{\rm comp}_{S,\gamma}(D)$ then for every $\beta \ge \gamma$ we have
{\rm comp}$_{S,\beta}(D) = \text{\rm comp}_{S,\gamma}(D)$
\smallskip

\noindent
\item[$(c)$]   there is an ordinal $\gamma = \gamma_S(D) <
\theta({\cP}(Y))$ such that {\rm comp}$_{S,\gamma}(D) = 
\text{\rm comp}_{S,\gamma +1}(D)$ and 
$\langle\text{\rm comp}_{S,\beta}(D):\beta
\le \gamma\rangle$ is strictly $\subset$-increasing.
\end{enumerate}
\medskip

\noindent
2) Assume AC$_S$.
\und{Then} for any filter $D$ on $Y$ we have $\gamma_S(D) \le \theta$
when $\theta := \text{\rm min}\{\lambda:\lambda$ a cardinal such that
cf$(\lambda) \ge \theta(S)\}$.

\noindent
3) Assume DC + AC$_S + |S \times S| = |S|$.  Then for any filter $D$
on $Y$ we have $\gamma_S(D) \le 1$ and comp$_{S,1}(D)$ is an $(\le
S)$-complete filter or is $\cP(Y)$; the latter holds iff $D$ is not
pseudo $(\le S)$-complete.

\noindent
4) Similarly to part (2) for ``$\in {\cS}$" but AC$_S$ is replaced by 
$S \in \mathscr{S} \Rightarrow \text{\rm AC}_S$ and 
$\theta = \text{\rm min}
\{\kappa:\kappa$ regular and $S \in \mathscr{S} \Rightarrow \kappa 
\ge \theta(\mathscr{S})\}$.
\end{observation}

\begin{remark}
Note that in part (2) of \ref{z0.55}, $\theta$ is regular and $\theta
\le \theta({}^{\omega >}S)$ but the inverse is not true, if $\theta(S)
= \aleph_0$ but holds if $\theta(S) > \aleph_0$.
\end{remark}

\begin{PROOF}{\ref{z0.55}}
 We prove the versions with $\mathscr{S}$, i.e. for (4).
Let $D_\gamma = \text{ comp}_{\in \mathscr{S},\gamma}(D)$ for $\gamma
\in \text{ Ord}$.

\noindent
1) Clause (a) is by the definition; clause (b) is proved by induction
on $\beta \ge \gamma$, for $\beta = \gamma$ this is trivial, for $\beta
= \gamma +1$ use the assumption and for $\beta > \gamma +1$ use the
definition and the induction hypothesis. 
As for clause (c) let $\gamma_* = \text{ min}\{\gamma \in
\text{ Ord} \cup \{\infty\}$; if 
$\gamma < \infty$ then $D_\gamma = D_{\gamma +1}\}$, 
so $\langle D_\gamma:\gamma \le \gamma_*\rangle$
is $\subset$-increasing continuous by clause (a), and by clause (b), $\langle
D_\gamma:\gamma \ge \gamma_*\rangle$ is constant.  Now define $h:\cP(Y)
\rightarrow \gamma_*$ by: $A \in D_{\gamma +1} \backslash D_\gamma \Rightarrow
h(A) = \gamma$ and $h(A) = 0$ when there is no such $\gamma$.  
So $h$ is onto $\gamma_*$
hence $\gamma_* < \theta(\cP(A))$ so $\gamma_*$ is as required on
$\gamma_S(D)$. 

\noindent
2) We prove also the relevant statement in part (4), so $S \in
\mathscr{S} \Rightarrow \text{ AC}_S \wedge \text{ cf}(\theta) \ge \theta(S)$.
Let $\gamma$ be an ordinal.

Let

\begin{equation*}
\begin{array}{clcr}
\cT^1_n = \{\Lambda:&\Lambda \text{ is a set of sequences of 
length } \le n, \\
  &\text{ closed under initial segments such that for every
non-maximal } \eta \in \Lambda \\
  &\text{ for some } S \in \mathscr{S} \text{ we have} \\
  &\eta \char 94 \langle s \rangle \in \Lambda \Leftrightarrow s \in S\}.
\end{array}
\end{equation*}

\begin{equation*}
\begin{array}{clcr}
{\cT}^2_{\gamma,n} = \{\bold x:&(a) \quad \bold x 
\text{ has the form } \langle Y_\eta,\gamma_\eta:\eta \in
\Lambda\rangle \\
  &(b) \quad \Lambda \in \cT^1_n \text{ and } Y_\eta \subseteq Y
\text{ for } \eta \in \Lambda\\ 
  &(c) \quad Y_\eta = \cap\{Y_{\eta \char 94 <s>}:s \text{ satisfies }
   \eta \char 94 \langle s \rangle \in \Lambda\} \text{ if } \eta \in
\Lambda \\
  &\qquad \text{ but } \eta \text{ is not } 
\triangleleft\text{-maximal in } \Lambda \\
  &(d) \quad \eta \triangleleft \nu \in \Lambda \Rightarrow 
\gamma_\nu < \gamma_\eta < 1 + \gamma \\
  &(e) \quad Y_\eta \in D \text{ if } \eta \in 
\Lambda \text{ is } \triangleleft\text{-maximal in } \Lambda \\
  &\qquad \text{ but } \ell g(\eta) < n\}
\end{array}
\end{equation*}

\[
\cT^2_n = \cup\{\cT^2_{\gamma,n}:\gamma \text{ is an ordinal}\}.
\]
\medskip

\noindent
Let $\bold n(\bold x) = n$ for the minimal possible $n$ such that
$\bold x \in {\cT}^2_n$ and let $\bold x = \langle Y^{\bold
x}_\eta,\gamma^{\bold x}_\eta:\eta \in \Lambda_{\bold x}\rangle$.

Let ${\cT}^3_\gamma = \cup\{{\cT}^2_{\gamma,n}:n < \omega\}$ and
let $<_*$ be the natural order on ${\cT}^3_\gamma:\bold x <_* \bold
y$ \und{iff} $n(\bold x) < n(\bold y),\Lambda_{\bold x} = \Lambda_{\bold y}
\cap {}^{n(\bold x)\ge}(\cup\{S:S \in \mathscr{S}\})$ 
and $(Y^{\bold x}_\eta,\gamma^{\bold x}_\eta) = 
(Y^{\bold y}_\eta,\gamma^{\bold y}_\eta)$ for $\eta
\in \Lambda_{\bold x}$.

Now 
\begin{enumerate}
\item[$\circledast$]  $A \in D_\gamma$ iff there
is an $\omega$-branch $\langle \bold x_n:n < \omega\rangle$ of 
$({\cT}^3_\gamma,<_*)$ such that $Y^{\bold x_n}_{<>} = A$.
\end{enumerate}
\medskip

\noindent
[Why?  We prove it by induction on the ordinal $\gamma$.  For
$\gamma=0$ and $\gamma$ limit this is obvious so assume we have it for
$\gamma$ and we shall prove it for $\gamma +1$.  

First assume $A \in D_{\gamma +1}$ and we shall find such
$\omega$-branch; if $A \in D_\gamma$ this is obvious, 
otherwise there are $S \in
\mathscr{S}$ and a sequence $\langle A_s:s \in S \rangle$
of members of $D_\gamma$ such that $A = \cap\{A_s:s \in S\}$.  So 
$X_s := \{\bar{\bold x}:\bar{\bold x}$ witness $A_s \in D_\gamma\}$ is 
well defined and non-empty by the induction hypothesis, clearly the
sequence $\langle X_s:s \in S\rangle$ exists, hence we can use 
AC$_S$ to choose $\langle \bar{\bold x}_s:s \in S\rangle$ satisfying
$\bar{\bold x}_s \in X_s$.

Now define $\bar{\bold x} = \langle \bold x_n:n < \omega\rangle$ as
follows:  $\Lambda_{\bold x_n} = \{\langle \rangle\} \cup \{\langle
s\rangle \char 94 \eta:\eta \in \Lambda_{\bold x_s,n-1}$ and $s \in
S\},\gamma^{\bold x_n}_{<>} = \cup\{\gamma^{\bold x_{s,n}}_{<>} +1:s \in
S\}$ and $Y^{\bold x}_{<>} = A$ and $Y^{\bold x_n}_{<s>\char 94 \eta} =
Y^{\bold x_{s,n-1}}_\eta$.  Now check.

Second, assume that there is such $\omega$-branch
$\langle \bold x_n:n < \omega\rangle$ of $(\cT^3_\gamma,<_*)$ such that
$Y^{\bold x_n}_{<>} = A$.  Let $S = \{\eta(0):\eta \in \Lambda_{\bold
x_1}\}$ so necessarily $S \in \mathscr{S}$.  For each $n < \omega$ and
$s \in S$ we define $\bold y_{n,s}$ as follows: 
$\Lambda^{\bold y_{n,s}}_s =
\{\nu:\langle s \rangle \char 94 \nu \in \Lambda_{\bold x_{n+1}}\}$
and for $\nu \in \Lambda^n_s$ let $\gamma^{\bold y_{n,s}}_\nu = \nu^{\bold
x_{n+1}}_{<s> \char 94 \nu}$ and $Y^{\bold y_{n,s}}_\nu = Y^{\bold
x_{n+1}}_{<s> \char 94 \nu}$.  Now clearly $\langle \bold y_{n,s}:n <
\omega\rangle$ is an $\omega$-branch of $(\cT^3_\gamma,\le_*)$ so by
the induction hypothesis $A^{\bold x_1}_{<s>} \in D$,
comp$_{S,\gamma}(D)$ and $Y^{\bold x_0}_{<>} = A = \cap\{Y^{\bold
x}_{<s>}:<> \in \Lambda_{\bold x_1}\} \in 
\text{ comp}_{S,\gamma+1}(D)$.  So we are done.]

Now toward a contradiction assume that $\gamma_S(D) > \theta$, so
there is $A \in D_{\theta +1} \backslash D_\theta$
hence here is an $\omega$-branch
$\langle \bold x_n:n < \omega\rangle$ of $\cT^3_\gamma$ witnessing
that $A \in D_{\theta +1}$, let $\Lambda = \cup\{\Lambda_{\bold x_n}:n
< \omega\}$ and $\gamma_\eta = \gamma^{\bold x_n}_\eta$ for every $n <
\omega$ large enough.  So $\Lambda$ is well founded (recalling $\eta
\triangleleft \nu \in \Lambda \Rightarrow \gamma_\eta > \gamma_\nu$)
and we can choose $\langle \gamma'_\eta:\eta \in \Lambda\rangle$
such that $\gamma'_\eta = \sup\{\gamma_\nu +1:\eta \triangleleft \nu
\in \Lambda$ and $\ell g(\nu) = \ell g(\eta) + 1\}$.  If $\gamma_{<>}
< \theta$ we are done otherwise let $\eta \in \Lambda$ be
$\triangleleft$-maximal such that $\gamma'_\eta \ge \theta$ hence $\eta
\triangleleft \nu \Rightarrow \gamma'_\nu < \theta$, so necessarily
$\gamma'_\eta = \theta = \cup\{\gamma'_\nu +1:\eta \triangleleft \nu
\in \Lambda,\ell g(\nu) = \ell g(\eta) +1\}$.  Let $S \in \mathscr{S}$
be such that $\eta \char 94 \langle s\rangle \in \Lambda
\Leftrightarrow s \in S$, so $\{\gamma'_{\eta \char 94 <s>}:s \in S\}$
is an unbounded subset of $\theta$ so cf$(\theta) \le \theta(S) <
\theta$.  This takes care of the first possibility for $\theta$ so the
second case is easier.

\noindent
3) It suffices to show that we can replace $\bold x \in \cT^2_2$ by
$\bold x \in \cT^2_1$.
\end{PROOF}

\begin{definition}
\label{z0.59} 
1) For a filter $D$ on a set $Y$ and a set $S$ let 
$\gamma_S(D)$ be as in clause (c) of the Observation \ref{z0.55}(1).

\noindent
1A) Similarly with ``$\in {\cS}$" instead $S$.

\noindent
2) $D$ is pseudo $(S,\gamma)$-complete if 
$\emptyset \notin \text{\rm comp}_{S,\gamma}(D)$. 

\noindent
2A) Similarly with ``$\in {\cS}$" instead $S$.
\end{definition}

\begin{observation}
\label{z0.79}  
1) If $h$ is a function from 
$S_1$ onto $S_2$ then $\theta(S_1) \ge \theta(S_2)$ and every [pseudo]
$(\le S_1)$-complete filter is a [pseudo] $(\le S_2)$-complete filter.
\end{observation}

\section {Commuting ranks}

The aim of this section is to sort out when two rank rk$_{D_1}$,
rk$_{D_2}$ do so called commute.
\bigskip

\noindent
\begin{definition}
\label{c3.1}  
Assume that $D_\ell$ 
is an $\aleph_1$-complete filter on $Y_\ell$ for $\ell=1,2$.  For $\iota
\in \{1,2,3,4,5\}$ we say $D_2$ does $\iota$-commute with 
$D_1$ \und{when}: $\boxplus_\iota  = 
\boxplus^\iota_{D_1,D_2}$ holds \und{where}:
\medskip

\noindent
\begin{enumerate}
\item[$\boxplus_1$]  if $A \in D_1$ and $\bar B = \langle B_s:s \in
A\rangle \in {}^A(D_2)$ \und{then} we can find $A_*,B_*$ such that:
$A_* \in D_1,B_* \in D_2$ and $A_* \times B_* \subseteq \cup\{\{s\}
\times B_s:s \in A\}$ so $A_* \subseteq A$
\smallskip

\noindent
\item[$\boxplus_2$]   if $A \in D_1$ and $\bar B = \langle B_s:s \in
A\rangle \in {}^A(D_2)$ and $J_2 = J[f_2,D_2]$ for some $f_2 \in
{}^{Y_2}\text{Ord}$ \und{then}  
we can find $A_*,B_*$ such that $A_* \in D_1,
B_* \in J^+_2$ and $A_* \times B_* \subseteq \cup\{\{s\} 
\times B_s:s \in A\}$ so $A_* \subseteq A$
\smallskip

\noindent
\item[$\boxplus_3$]   if $A \in D_1$ and $\bar B = \langle B_s:s \in
A \rangle \in {}^A(D_2)$ and $J_1 = J[f_1,D_1]$ for some $f_1 
\in {}^{Y_1}\text{Ord}$ \und{then} we can find
$A_*,B_*$ such that $A_* \in J^+_1,A_* \subseteq A,B_* \in D_2$ and $s \in
A_* \Rightarrow B_* \subseteq B_s$
\smallskip

\noindent
\item[$\boxplus_4$]  if $A \in D_1$ and $\bar B = \langle B_s:s \in
 A\rangle \in {}^A(D_2)$ and $\bar J^1 = \langle J^1_t:t \in Y_2\rangle$
 satisfies $J^1_t \in \{J[f,D_1]:f \in{}^{Y_1}\text{Ord}\}$ and $J_2
 \in\{J[f,D_2]:f \in {}^{Y_2}\text{Ord}\}$ \und{then} we can find
$A_*,B_*$ such that $B_* \in J^+_2$ and $t \in B_* \Rightarrow
A_* \in (J^1_t)^+$ and $(s,t) \in A_* \times B_* \Rightarrow s \in A \wedge t
 \in B_s$ hence $A_* \subseteq A,A_* \in D^+_1$
\smallskip

\noindent
\item[$\boxplus_5$]  like $\boxplus_4$ but we omit the sequence $\bar
J^1$ and the demand on $A_*$ is $A_* \in D^+_1$.
\end{enumerate}
\end{definition}

\begin{remark}
\label{c3.2}  
1) These are seemingly not commutative relations.

\noindent
2) We shall first give a consequence and then give sufficient
   conditions.

\noindent
3) We intend to generalize to systems (see \ref{m4.3} and \ref{m4.9}).

\noindent
4) Can we below use ``$D_\ell \in \text{ FIL}_{\text{pcc}}(Y_1)$, see
   Definition \ref{z0.32}?  Yes, but only when we do not use $D+A,A \in D^+$.
\end{remark}

\begin{claim}
\label{c3.3}
{\rm rk}$_{D_1}(f) \le \text{\rm rk}_{D_2}(g)$ \und{when}:
\medskip

\noindent
\begin{enumerate}
\item[$\oplus$]   $(a) \quad D_\ell \in 
\text{\rm FIL}_{\text{cc}}(Y_\ell)$ for $\ell=1,2$
\smallskip

\noindent
\item[${{}}$]   $(b) \quad \bar g = \langle g_t:t \in Y_2\rangle$
\smallskip

\noindent
\item[${{}}$]   $(c) \quad g_t \in {}^{Y_1}\text{\rm Ord}$
\smallskip

\noindent
\item[${{}}$]   $(d) \quad g \in {}^{Y_2}\text{\rm Ord}$ is defined by
$g(t) = \text{\rm rk}_{D_1}(g_t)$
\smallskip

\noindent
\item[${{}}$]   $(e) \quad \bar f = \langle f_s:s \in Y_1\rangle$
\smallskip

\noindent
\item[${{}}$]   $(f) \quad f_s \in {}^{Y_2}\text{\rm Ord}$ is
defined by $f_s(t) = g_t(s)$
\smallskip

\noindent
\item[${{}}$]   $(g) \quad f \in {}^{Y_1}\text{\rm Ord}$ is defined
by $f(s) = \text{\rm rk}_{D_2}(f_s)$
\smallskip

\noindent
\item[$\boxplus$]   $(a) \quad D_2$ does 2-commute with $D_1$
\smallskip

\noindent
\item[${{}}$]   $(b) \quad$ AC$_{Y_1}$ holds.
\end{enumerate}
\end{claim}

\begin{remark}  In order not to use DC in the proof we should consider
$\infty$ as a member of Ord in clauses (d),(g) of $\boxplus$.
\end{remark}

\begin{PROOF}{\ref{c3.3}}  
We prove by induction on the ordinal $\zeta$ that
\medskip

\noindent
\begin{enumerate}
\item[$\boxdot_\zeta$]    if $\oplus + \boxplus$ above hold for
$D_1,D_2,f,g,\bar f,\bar g$ and rk$_{D_1}(f) \ge \zeta$ then
rk$_{D_2}(g) \ge \zeta$.
\end{enumerate}
\medskip

\noindent
The case $\zeta=0$ is trivial and the case $\zeta$ a limit ordinal
follows by the induction hypothesis.  So assume that $\zeta = \xi+1$.

Let
\medskip

\noindent
\begin{enumerate}
\item[$(*)_1$]   $A := \{s \in Y_1:f(s) > 0\}$.
\end{enumerate}
\medskip

\noindent
As we are assuming rk$_{D_1}(f) > \xi \ge 0$ by \ref{z0.23}(4) necessarily
\medskip

\noindent
\begin{enumerate}
\item[$(*)_2$]   $A \in D_1$.
\end{enumerate}
\medskip

\noindent
For each $s \in A,f(s) > 0$ so applying clause (g) of $\oplus$ we get
\medskip

\noindent
\begin{enumerate}
\item[$(*)_3$]   rk$_{D_2}(f_s) > 0$ when $s \in A$ 
\end{enumerate}
\medskip

\noindent
hence
\medskip

\noindent
\begin{enumerate}
\item[$(*)_4$]   $B_s := \{t \in Y_2:f_s(t) > 0\}$ belongs to $D_2$
when $s \in A$.
\end{enumerate}
\medskip

\noindent
So $\langle B_s:s \in A\rangle \in {}^A(D_2)$.
Recall (see $\boxplus(a)$ of the assumption) that $D_2$ does 2-commute
with $D_1$, apply it to $A,\langle B_s:s \in A\rangle,J_2 := J[g,D_2]$; so
we can find $A_*,B_*$ such that
\medskip

\noindent
\begin{enumerate}
\item[$(*)_5$]   $(a) \quad A_* \in D_1$ (and $A_* \subseteq A$)
\smallskip

\noindent
\item[${{}}$]   $(b) \quad B_* \in J^+_2$ recalling $J_2 = J[g,D_2]$
so $B_* \in D^+_2$ and (by Definition \ref{z0.27})
\newline

\hskip25pt  rk$_{D_2+B_*}(g) = \text{ rk}_{D_2}(g)$
\smallskip

\noindent
\item[${{}}$]    $(c) \quad (s,t) \in A_* \times B_* \Rightarrow s
\in A \wedge t \in B_s$.
\end{enumerate}
\medskip

\noindent
Now by the present assumption of $\boxdot_\zeta$ we have
\medskip

\noindent
\begin{enumerate}
\item[$(*)_6$]   rk$_{D_1}(f) \ge \zeta = \xi +1$.
\end{enumerate}
\medskip

\noindent
Hence by the definition of rk and \ref{z0.23}(2) we can find $f'$ such that
\medskip

\noindent
\begin{enumerate}
\item[$(*)_7$]   $(a) \quad f' \in {}^{Y_1}\text{Ord}$ 
and rk$_{D_1}(f') \ge \xi$
\smallskip

\noindent
\item[${{}}$]    $(b) \quad f' <_{D_1} f$ 
\smallskip

\noindent
\item[${{}}$]   $(c) \quad$ by $(*)_1$ \wilog \,
$s \in A \Rightarrow f'(s) < f(s)$.
\end{enumerate}
\medskip

\noindent
For each $s \in A$, clearly $f'(s) < f(s) = 
\text{\rm rk}_{D_2}(f_s) \le \text{ rk}_{D_2+B_*}(f_s)$, 
by \ref{z0.24}(2), clause (g) of $\oplus$ and $D_2 \subseteq D_2 +
B_*$ respectively, hence by \ref{z0.23}(2) for each 
$s \in Y_1$ there is a function $f'_s$ such that
\medskip

\noindent
\begin{enumerate}
\item[$(*)_8$]   $(a) \quad f'_s \in {}^{(Y_2)}\text{Ord}$,
\smallskip

\noindent
\item[${{}}$]   $(b) \quad f'_s < f_s \text{ mod } D_2$ if $s \in A$
and $t \in Y_2 \Rightarrow f'_s(t) < f_s(t) \vee f'_s(t) = 0 = f_s(t)$
\smallskip

\noindent
\item[${{}}$]   $(c) \quad$ rk$_{D_2+B_*}(f'_s) = f'(s)$; may
require this only for $s \in A$.
\end{enumerate}
\medskip

\noindent
As $Y_1 \in \bold C$ by $\boxplus(b)$ of the assumption, clearly
\medskip

\noindent
\begin{enumerate}
\item[$(*)^+_8$]   there is such a sequence 
$\bar f' = \langle f'_s:s \in Y_1\rangle$.
\end{enumerate}
\medskip

\noindent
As $s \in A_* \wedge t \in B_* \Rightarrow f_s(t) 
> 0$, see $(*)_4 + (*)_5$, clearly
\medskip

\noindent
\begin{enumerate}
\item[$(*)_9$]   if $s \in A_*$ and $t \in B_*$ then $f'_s(t) < f_s(t)$.
\end{enumerate}
\medskip

\noindent
We now define $\bar g' = \langle g'_t:t \in Y_2 \rangle$ by
\medskip

\noindent
\begin{enumerate}
\item[$(*)_{10}$]   $g'_t(s) = f'_s(t)$ for $s \in Y_1,t \in Y_2$ 
so $g'_t \in {}^{Y_1}\text{Ord}$.
\end{enumerate}
\mn
So
\medskip

\noindent
\begin{enumerate}
\item[$(*)_{11}$]   $s \in A_* \wedge t \in B_* 
\Rightarrow g'_t(s) = f'_s(t) < f_s(t) = g_t(s)$
\end{enumerate}
\medskip

\noindent
hence (recalling $A_* \in D_1$ by $(*)_5(a)$ and \ref{z0.23}(3))
\medskip

\noindent
\begin{enumerate}
\item[$(*)_{12}$]   if $t \in B_*$ then $g'_t <_{D_1} g_t$ hence
rk$_{D_1}(g'_t) < \text{ rk}_{D_1}(g_t)$.
\end{enumerate}
\medskip

\noindent
Define $g' \in {}^{(Y_2)}\text{Ord}$ by $g'(t) := 
\text{ rk}_{D_1}(g'_t)$ hence (recalling rk$_{D_1}(g_t) = g(t)$)
\medskip

\noindent
\begin{enumerate}
\item[$(*)_{13}$]   $g' < g$ mod $D_2 + B_*$.
\end{enumerate}
\medskip

\noindent
Note that here $D_1 + A_* = D_1$, (though not so when we shall prove
\ref{c3.14}). 

Now we apply the induction hypothesis to $g',f',\bar g' := \langle g'_t:t \in
Y_2\rangle,\bar f' := \langle f'_s:s \in Y_1\rangle,
D_1 + A_*,D_2 + B_*$ and $\xi$ and get
\medskip

\noindent
\begin{enumerate}
\item[$(*)_{14}$]   $\xi \le \text{ rk}_{D_2+B_*}(g')$.
\end{enumerate}
\medskip

\noindent
[Why is this legitimate?   First, obviously clauses (a),(b) of
$\boxplus$ holds, second, we have to check that clauses
(a)-(g) of $\oplus$ hold in this instance.
\bigskip

\noindent
\und{Clause (a)}:  First ``$D_1 + A_* \in \text{ FIL}_{\text{cc}}(Y_1)$" as
we assume $D_1 \in \text{ FIL}_{\text{cc}}(Y_1)$ and $A_* \in D_1$,
see $(*)_5(a)$, actually $A_* \in D^+_1$ suffice (used in proving 
\ref{c3.14}).
\newline

\noindent
Second, ``$D_2 + B_* \in \text{ FIL}_{\text{cc}}(Y_2)$" as
$D_2 \in \text{ FIL}_{\text{cc}}(Y_2)$ and $B_* \in D^+_2$ by
$(*)_5(b)$.
\bigskip

\noindent
\und{Clause (b)}:  ``$\bar g' = \langle g'_t:t \in Y_2\rangle$" by our
choice.
\bigskip

\noindent
\und{Clause (c)}:  ``$g'_t \in {}^{Y_1}\text{Ord}$" by $(*)_{10}$.
\bigskip

\noindent
\und{Clause (d)}:  ``$g' \in {}^{Y_2}\text{Ord}$ is defined by $g'(t) =
\text{ rk}_{D_1}(g'_t)$" by its choice after $(*)_{12}$.
\bigskip

\noindent
\und{Clause (e)}:  ``$\bar f' = \langle f'_s:s \in Y\rangle$" by our
choice in $(*)^+_8$.
\bigskip

\noindent
\und{Clause (f)}:  ``$f'_s \in {}^{Y_2}\text{Ord}$ is defined by
$f'_s(t) = g'_t(s)$ holds by $(*)_{10}$.
\bigskip

\noindent
\und{Clause (g)}:  ``$f' \in {}^{Y_1}\text{Ord}$ is defined by
$f'(s) = \text{ rk}_{D_2+B_*}(f'_s)$" holds by $(*)_7(a) + (*)_8(c)$.

Now $\boxdot_\xi$, the induction hypothesis, assumes
``rk$_{D_1+A_*}(f') \ge \xi$" which holds by $(*)_7(a) + (*)_5(a)$,
actually $A_* \in D^+_1$ suffice here and
its conclusion is $\xi \le \text{ rk}_{D_2+B_*}(g')$ as promised in
$(*)_{14}$.] 

Next
\medskip

\noindent
\begin{enumerate}
\item[$(*)_{15}$]   $\xi < \text{ rk}_{D_2}(g)$.
\end{enumerate}
\medskip

\noindent
[Why?  
\medskip

\noindent
\begin{enumerate}
\item[$\bullet_1$]   $\xi \le \text{ rk}_{D_2+B_*}(g')$ by $(*)_{14}$ 
\smallskip

\noindent
\item[$\bullet_2$]   rk$_{D_2+B_*}(g') < \text{ rk}_{D_2+B_*}(g)$ by
$(*)_{13}$ and \ref{z0.23}(3)
\smallskip

\noindent
\item[$\bullet_3$]   rk$_{D_2+B_*}(g) = \text{ rk}_{D_2}(g)$ by $(*)_5(b)$.
\end{enumerate}
\medskip

\noindent
Together $(*)_{15}$ holds.]

So
\begin{enumerate}
\item[$(*)_{16}$]   $\zeta = \xi +1 \le \text{ rk}_{D_2}(g)$
\end{enumerate}
as promised.  Together we are done.  
\end{PROOF}

\begin{claim}
\label{c3.11}  
Assume $D_\ell \in \text{\rm FIL}_{\text{cc}}(Y_\ell)$ for $\ell=1,2$.
\newline

If $D_2$ does $\iota_1$-commute with $D_1$ \und{then} $D_2$ does 
$\iota_2$-commute with $D_1$ when $(\iota_1,\iota_2) = (1,2),(1,3),(2,4),(1,4),(1,5),(4,5)$.
\end{claim}

\begin{PROOF}{\ref{c3.11}}
Obvious for (4,5) use \ref{z0.29}(3).  
\end{PROOF}

\begin{claim}
\label{c3.8}  
Assume $D_\ell \in \text{\rm FIL}_{\text{cc}}(Y_\ell)$ 
for $\ell=1,2$.  If at least one of the
following cases occurs, \und{then} $D_2$ does 1-commute (hence
2-commute) with $D_1$.
\medskip

\noindent
\und{Case 1}:  $D_2$ is $|Y_1|^+$-complete.
\medskip

\noindent
\und{Case 2}:  $D_1$ is an ultrafilter which is $|Y_2|^+$-complete
\medskip

\noindent
\und{Case 3}:  $D_1,D_2$ are ultrafilters and if $\bar A = \langle
A_t:t \in Y_2\rangle \in {}^{Y_2}(D_1)$ \und{then} for some $A_* \in
D_1$ we have $\{t:A_t \supseteq A_*\} \in D_2$.
\end{claim}

\begin{PROOF}{\ref{c3.8}}  
So let $A \in D_1$ and $\langle B_s:s \in A\rangle
\in{}^A(D_2)$ be given.
\medskip

\noindent
\und{Case 1}:  Let $A_* = A$ and $B_* = \cap\{B_s:s \in A\}$, so $A_*
\in D_1$ by an assumption and $B_* \in D_2$ as we assume $\{B_s:s
\in A\} \subseteq D_2$ and $D_2$ is $|Y_1|^+$-complete (and
necessarily $|A| \le |Y_1|$).
\medskip

\noindent
\und{Case 2}:  For each $t \in Y_2$ let $A'_t := \{s \in Y_1:s \in A$
and $t \in B_s\}$ and let $A''_t$ be the unique member of $\{A'_t,Y_1
\backslash A'_t\} \cap D_1$, recalling $D_1$ is an ultrafilter on
$Y_1$.  Clearly the functions $t \mapsto A'_t$ and $t \mapsto A''_t$
are well defined hence the sequences $\langle A'_t:t \in
Y_2\rangle,\langle A''_t:t \in Y_2\rangle$ exist and $\{A''_t:t \in
Y_2\} \subseteq D_1$.  

As $D_1$ is $|Y_2|^+$-complete necessarily $A_*
:= \cap\{A''_t:t \in Y_2\} \cap A$ belongs to $D_1$, and clearly $A_*
\subseteq A$.  Let $B_* = \{t \in Y_2:A''_t = A'_t\}$.

So now choose
any $s_* \in A_*$ (possible as $A_* \in D_1$ implies $A_* \ne
\emptyset$) so $B_{s_*} \in D_2$ and $t \in B_{s_*} \Rightarrow s_*
\in A'_t \Rightarrow s_* \in A'_t \cap A_* \Rightarrow A'_t \cap A_*
 \ne \emptyset \Rightarrow A''_t = A'_t \Rightarrow t \in B_*$ so $B_{s_*}
\subseteq B_*$ but $B_{s_*} \in D_2$ hence $B_* \in D_2$. 
So $A_*,B_*$ are as required.
\medskip

\noindent
\und{Case 3}:  

Like Case 2.  
\end{PROOF}

\begin{claim}
\label{c3.10}  
Assume {\rm AC}$_{\cP(Y_2)}$.

\noindent
1) Assume $D_1 \in \text{\rm FIL}_{\text{cc}}(Y_1)$ and 
$D_2 \in \text{\rm FIL}_{\text{cc}}(Y_2)$.

Then $D_2$ does 3-commute with $D_1$ \und{when} $D_1$ is $(\le 
{\cP}(Y_2))$-complete.

\noindent
2) In part (1) if $E \subseteq D_2$ is 
$(D_2,\subseteq)$-cofinal, it suffices to assume $D_1$ is 
$(\le E)$-complete.
\end{claim}

\begin{remark}
For part (1) in the definition of $(\le \cP(Y_2))$-complete 
we can use just partitions, but not so in part (2).
\end{remark}

\begin{PROOF}{\ref{c3.10}}  
1)  So let $A \in D_1$ and $\bar B = \langle B_s:s \in
A\rangle \in {}^A(D_2)$ and $J_1 = J[f_1,D_1]$ for some $f_1 \in
 {}^Y\text{Ord}$ be given.  So $s \mapsto B_s$
is a function from $A \in D_1$ to $D_2 \subseteq {\cP}(Y_2)$ hence
as AC$_{\cP(Y_2)}$ is assumed recalling that by \ref{z0.29}(1) 
the ideal $J_1$ on $Y_1$ 
is $(\le {\cP}(Y_2))$-complete, there is $B_* \in D_2$ such 
that $A_* := \{s \in A:B_s =B_*\} \in J^+_1$.  Clearly $A_*,B_*$ are
as required.

\noindent
2) For $B \in E$ let $A^*_B = \{s \in A:B \subseteq B_s\}$, so
clearly $\langle A^*_B:B \in E\rangle$ is a sequence of subsets of
$A \in D_1$ with union $A$, so again by \ref{z0.29}(1) for 
some $B_* \in E$ the set $A_* :=
 \{s \in A:B_* \subseteq B_s\}$ belongs to $J^+_1$, so we are done. 
\end{PROOF}

\begin{claim}
\label{c3.14}  {\rm rk}$_{D_1}(f) \le \text{\rm rk}_{D_2}(g)$
\und{when}:
\medskip

\noindent
\begin{enumerate}
\item[$\oplus$]   as in \ref{c3.3}
\end{enumerate}
but we replace clause $(\boxplus)$ there by
\medskip

\noindent
\begin{enumerate}
\item[$\boxplus'$]  $(a) \quad D_2$ does 4-commute with $D_1$
\smallskip

\noindent
\item[${{}}$]  $(b) \quad$ {\rm AC}$_{Y_1}$ holds.
\end{enumerate}
\end{claim}

\begin{PROOF}{\ref{c3.14}}
We repeat the proof of \ref{c3.3} but:
\medskip

\noindent
\und{First change}: we replace $(*)_5$ and the paragraph before it by
the following:

So $\bar B = \langle B_s:s \in A\rangle \in {}^A(D_2)$.

Recall that $D_2$ does 4-commute with $D_1$, apply this to $A,\langle
B_s:s \in A\rangle,\bar J^1 = \langle J^1_t:t \in Y_2\rangle$ where $J^1_t
:= J[g_t,D_1],J_2 := J[g,D_2]$ and we get $A_*,B_*$ such that:
\medskip

\noindent
\begin{enumerate}
\item[$(*)'_5$]  $(a) \quad A_* \in D^+_1$ and $A_* \subseteq A$
\smallskip

\noindent
\item[${{}}$]   $(b) \quad B_* \in J^+_2$ hence $B_* \in D^+_2$ and
rk$_{D_2+B_*}(g) = \text{ rk}_{D_2}(g)$
\smallskip

\noindent
\item[${{}}$]   $(c) \quad (s,t) \in A_* \times B_* \Rightarrow s
\in A \wedge t \in B_s$
\smallskip

\noindent
\item[${{}}$]   $(d) \quad$ if $t \in B_*$ then $A_* \in (J^1_t)^+$
hence 
\newline

\hskip25pt $t \in B_* \Rightarrow \text{ rk}_{D_1+A_*}(g_t) 
= \text{ rk}_{D_1}(g_t) = g(t)$.
\end{enumerate}
\medskip

\noindent
\und{Second change}: we replace $(*)_{12}$ and the line before, the line
after it and $(*)_{13}$ by:

Define $g' \in {}^{Y_2}\text{Ord}$ by $g'(t) = \text{ rk}_{D_1 + A_*}(g'_t)$.  

Now
\medskip

\noindent
\begin{enumerate}
\item[$(*)_{12}'$]   if $t \in B_*$ then
\begin{enumerate}
\item[$(a)$]   $g'_t <_{D_1+A_*} g_t$, by $(*)_{11}$
\smallskip

\noindent
\item[$(b)$]   rk$_{D_1+A_*}(g'_t) < \text{ rk}_{D_1+A_*}(g_t)$
by (a) and \ref{z0.23}(3),
\smallskip

\noindent
\item[$(c)$]  rk$_{D_1+A_*}(g_t) = \text{ rk}_{D_1}(g_t)$ 
recalling $(*)'_5(d)$ and $J^1_t = J[g_t,D_1]$ hence
\smallskip

\noindent
\item[$(d)$]   rk$_{D_1}(g_t) = g(t)$ by clause (d) of $\oplus$
\sn
\item[$(e)$]   rk$_{D_1+A_*}(g_t) = g(t)$ by (c), (d) above hence
\sn
\item[$(f)$]  $g'(t) < g(t)$ by the choice of $g'$, clause (b) and
clause (e).
\end{enumerate}
\end{enumerate}
\mn
Hence by $(*)'_{12}(f)$ we have
\mn
\begin{enumerate}
\item[$(*)'_{13}$]  $g' < g$ mod $D_2 + B_*$.
\end{enumerate}
\medskip

\noindent
Concerning the rest, we quote $(*)_5(b)$ twice but $(*)'_5(b)=
(*)_5(b)$, and quote $(*)_5(a)$ twice but noted there that $(*)'_5(a)$
suffice and $g'$ is defined before $(*)'_{12}$ rather than apply $(*)_{12}$.
\end{PROOF}

\section{Rank systems and A Relative of GCH} 

To phrase our theorem we need to define the framework.

\begin{definition}
\label{m4.3}
Main Definition:   We say that $\bold p = 
(\bbD,\text{rk},\Sigma,\bold j,\mu) = (\bbD_{\bold p},
\text{rk}_{\bold p},\Sigma_{\bold p},\bold j_{\bold p},\mu_{\bold p})$
is a weak (rank) 1-system \und{when}:
\medskip

\noindent
\begin{enumerate}
\item[$(a)$]    $\mu$ is singular
\smallskip

\noindent
\item[$(b)$]   each $\bold d \in \bbD$ is (or just we can compute
from it) a pair $(Y,D) = (Y_{\bold d},D_{\bold d}) = (Y[\bold
d],D_{\bold d}) = (Y_{\bold p,\bold d},D_{\bold p,\bold d})$ such that:
\begin{enumerate}
\item[$(\alpha)$]   $\theta(Y_{\bold d}) < \mu$, on $\theta(-)$
see Definition \ref{z0.15}
\smallskip

\noindent
\item[$(\beta)$]   $D_{\bold d}$ is a filter on $Y_{\bold d}$
\end{enumerate}
\item[$(c)$]   for each $\bold d \in \bbD$, a definition
of a function rk$_{\bold d}(-)$ with domain ${}^{Y[\bold d]}\text{Ord}$ and
range $\subseteq \text{ Ord}$, that is rk$_{\bold p,\bold d}(-)$ or 
rk$^{\bold p}_{\bold d}(-)$ 
\smallskip

\noindent
\item[$(d)$]   $(\alpha) \quad \Sigma$ is a function with
domain $\bbD$ such that $\Sigma({\bold d}) \subseteq \bbD$
\smallskip

\noindent
\item[${{}}$]   $(\beta) \quad$ if ${\bold d} \in \bbD$
and $\bold e \in \Sigma(\bold d)$ then $Y_{\bold e} = 
Y_{\bold d}$ [natural to add $D_{\bold d} \subseteq D_{\bold e}$, 
\newline

\hskip25pt this is not demanded but see \ref{m4.9}(2)]
\smallskip

\noindent
\item[$(e)$]   $(\alpha) \quad \bold j$ is a function from $\bbD$
onto cf$(\mu)$
\smallskip

\noindent
\item[${{}}$]  $(\beta) \quad$ let $\bbD_{\ge i} =
\{\bold d \in \bbD:\bold j(\bold d) \ge i\}$ and $\bbD_i = 
\bbD_{\ge i} \backslash \bbD_{i+1}$
\smallskip

\noindent
\item[${{}}$]   $(\gamma) \quad \bold e \in \Sigma(\bold d)
\Rightarrow \bold j(\bold e) \ge \bold j(\bold d)$
\smallskip

\noindent
\item[$(f)$]   for every $\sigma < \mu$ for some 
$i < \text{ cf}(\mu)$, if $\bold d \in \bbD_{\ge i}$, \und{then}
$\bold d$ is $(\bold p,\le \sigma)$-complete where:
\begin{enumerate}
\item[$(*)$]    we say that $\bold d$ is
$(\bold p,\le X)$-complete (or $(\le X)$-complete for $\bold p$) 
when: if $f \in {}^{Y[\bold d]}\text{Ord}$ 
and $\zeta = \text{ rk}_{\bold d}(f)$ and $\langle A_j:j \in X\rangle$ 
a partition\footnote{as long as $\sigma$ is a well
ordered set it does not matter whether we use a partition or just a
covering, i.e. 
$\cup\{A_j:j \in \sigma\} = Y_{\bold d}$} of $Y_{\bold d}$, \und{then}
for some $\bold e \in \Sigma(\bold d)$ and $j < \sigma$ we have $A_j
\in D_{\bold e}$ and $\zeta = \text{ rk}_{\bold e}(f)$; so this is not
the same as ``$D_{\bold d}$ is $(\le X)$-complete"; we define $(\bold
p,|X|^+)$-complete, i.e. $(\bold p,<|X|^+)$-complete similarly
\end{enumerate} 
\item[$(g)$]   no hole\footnote{we may use another function $\Sigma$
here, as in natural examples here we use $\Sigma(\bold d) = \{\bold d\}$ and
not so in clause (f)}: if rk$_{\bold d}(f) > \zeta$ then for some pair
$(\bold e,g)$ we have: $\bold e \in \Sigma(\bold d)$ and 
$g <_{D[\bold e]} f$ and rk$_{\bold e}(g) = \zeta$
\smallskip

\noindent
\item[$(h)$]   if $f=g+1$ mod $D_{\bold d}$ \und{then} rk$_{\bold d}(f) =
\text{ rk}_{\bold d}(g)+1$
\smallskip

\noindent
\item[$(i)$]    if $f \le g$ mod $D_{\bold d}$ \und{then} rk$_{\bold d}(f) \le
\text{ rk}_{\bold d}(g)$.
\end{enumerate}
\end{definition}

\begin{definition}
\label{m4.4}  
1) We say $\bold p = (\bbD,\text{rk},\Sigma,\bold j,\mu)$ is 
a weak (rank) 2-system, (if we write system we mean 2-system)
\und{when} in \ref{m4.3} we replace clauses (d),(f),(g) by:
\medskip

\noindent
\begin{enumerate}
\item[$(d)'$]   $(\alpha) \quad \Sigma$ is a function with domain
$\bbD$
\smallskip

\noindent
\item[${{}}$]   $(\beta) \quad$ for $\bold d \in \bbD$ we have 
$\Sigma(\bold d) \subseteq \{(\bold e,h):\bold e \in 
\bbD_{\ge \bold j(\bold d)}$ and $h:Y_{\bold e} \rightarrow Y_{\bold d}\}$;

\hskip25pt writing $\bold e \in \Sigma(\bold d)$ means
then $(\bold e,h) \in \Sigma(\bold d)$ for some function $h$
\sn
\item[$(f)'$]   for every $\sigma < \mu$ for some 
$i < \text{ cf}(\mu)$, if $\bold d \in \bbD_{\ge i}$, \und{then}
$\bold d$ is $(\bold p,\le \sigma)$-complete where:
\begin{enumerate}
\item[$(*)$]    we say that $\bold d$ is $(\bold p,\le X)$-complete 
(for $\bold p$) when: if $f \in {}^{Y[\bold d]}\text{Ord}$ 
and $\zeta = \text{ rk}_{\bold d}(f)$ and $\langle A_j:j \in X\rangle$ 
a partition\footnote{as long as $\sigma$ is a well
ordered set it does not matter whether we use a partition or just a
covering, i.e. 
$\cup\{A_j:j \in \sigma\} = Y_{\bold d}$} of $Y_{\bold d}$, \und{then}
for some $(\bold e,h) \in \Sigma(\bold d)$ and $j < \sigma$ we have 
$h^{-1}(A_j) \in D_{\bold e}$ and $\zeta = \text{ rk}_{\bold e}(f \circ h)$;
we define ``$(\bold p,|X|^+)$-complete" similarly
\end{enumerate}
\item[$(g)'$]   no hole: if rk$_{\bold d}(f) > \zeta$ then for some
$(\bold e,h) \in \Sigma(\bold d)$ and $g \in {}^{Y[\bold
e]}\text{Ord}$ we have $g < f \circ h$ mod $D_{\bold e}$ and
rk$_{\bold e}(g) = \zeta$.
\end{enumerate}
\end{definition}

\begin{dc}
\label{m4.4d}
Let $\bold p$ be a weak rank 1-system; we can 
define $\bold q$ and prove it is a weak
rank 2-system by $\bbD_{\bold q} = \bbD_{\bold p}$, rk$_{\bold q} =
 \text{ rk}_{\bold p},\Sigma_{\bold q}(\bold d) = \{(\bold
 e,\text{id}_{Y[\bold d]}):\bold e \in \Sigma_{\bold p}(\bold d)\},\bold
 j_{\bold q} = \bold j_{\bold p},\mu_{\bold q} = \mu_{\bold p}$.
\end{dc}

\begin{convention}
\label{m4.5} 
1) We use $\bold p$ only for systems as in
Definition \ref{m4.3} or \ref{m4.4}.

\noindent
2) We may not distinguish $\bold p$ and $\bold q$ in \ref{m4.4d} so
deal only with 2-systems.
\end{convention}

\begin{remark}
\label{m4.5d}  
The following is an alternative to Definition \ref{m4.4}.  As in
\ref{m4.3} we can demand $\bold e \in \Sigma(\bold d) \Rightarrow Y_{\bold e} =
Y_{\bold d}$ but for every $\bold d$ we have a family $\cE_{\bold d}$,
i.e. the function $\bold d \mapsto \cE_{\bold d}$ is part of $\bold p$
and make the following additions and changes: 
\medskip

\noindent
\begin{enumerate}
\item[$(\alpha)$]   ${\cE}_{\bold d}$ is a family of equivalence 
relations on $Y_{\bold d}$
\smallskip

\noindent
\item[$(\beta)$]   we replace ${}^{Y[\bold d]}\text{Ord}$ by $\{f \in
{}^{Y[\bold d]}\text{Ord}$: eq$(f) := \{(s,t):s,t \in Y_{\bold d}$ and $f(s)
= f(t)\} \in {\cE}_{\bold d}\}$
\smallskip

\noindent
\item[$(\gamma)$]   if $E_1,E_2$ are equivalence relations on $Y_{\bold
d}$ such that $E_2$ refines $E_1$ \und{then} $E_2 \in {\cE}_{\bold
d} \Rightarrow E_1 \in {\cE}_{\bold d}$
\smallskip

\noindent
\item[$(\delta)$]    if $\bold e \in \Sigma(\bold d)$ then $Y_{\bold e} =
Y_{\bold d}$ and ${\cE}_{\bold d} \subseteq {\cE}_{\bold e}$.
\end{enumerate}
\end{remark}

\begin{definition}
\label{m4.6}  
For $\iota = 1,2$ we say that $\bold p = 
(\Bbb D,\text{rk},\Sigma,\bold j,\mu)$ is a strict $\iota$-system
\und{when} it satisfies clauses (a)-(i) from \ref{m4.3} or from \ref{m4.4}
and
\medskip

\noindent
\begin{enumerate}
\item[$(j)$]   for every $\bold d \in \bbD$ and $\zeta,\xi,f,j_0$ satisfying
$\boxplus$ below, there\footnote{can we make $j$ depend on $f$ (and a
partition of) $\bold d$?  Anyhow later we use $\bold d' \in
\Sigma(\bold d)$, if $\{\bold d\} \ne \Sigma(\bold d)$?  Also so
$\iota = 1,2$ may make a difference.} 
is $j < \text{ cf}(\mu)$ such that: there are no $\bold e,g$
satisfying $\oplus$ below, where:
\begin{enumerate}
\item[${{}}$]  $\oplus \quad \bullet_1 \quad \bold e \in \bbD_{\ge j}$
\smallskip

\noindent
\item[${{}}$]  $\qquad \bullet_2 \quad g \in  {}^{Y[\bold e]}\zeta$
\smallskip

\noindent
\item[${{}}$]  $\qquad \bullet_3 \quad \{g(t):t \in Y_{\bold e}\}
\subseteq [\xi,\zeta_*)$ for some $\zeta_* < \zeta$
\smallskip

\noindent
\item[${{}}$]  $\qquad \bullet_4 \quad j \ge j_0$
\smallskip

\noindent
\item[${{}}$]  $\qquad \bullet_5 \quad$ rk$_{\bold e}(g) \ge \zeta$,
\smallskip

\noindent
\item[${{}}$]  $\boxplus \quad \bullet_1 \quad f \in {}^{Y[\bold d]}\xi$
\smallskip

\noindent
\item[${{}}$]  $\qquad \bullet_2 \quad$ rk$_{\bold d}(f) = \zeta$
\smallskip

\noindent
\item[${{}}$]  $\qquad \bullet_3 \quad \xi < \zeta$
\smallskip

\noindent
\item[${{}}$]  $\qquad \bullet_4 \quad$ cf$(\mu) = \text{ cf}(\zeta)$
\smallskip

\noindent
\item[${{}}$]  $\qquad \bullet_5 \quad j_0 < \text{ cf}(\mu)$
\smallskip

\noindent
\item[${{}}$]  $\qquad \bullet_6 \quad s \in Y_{\bold d} 
\wedge \bold e' \in \Bbb D_{\ge j_0} \Rightarrow  
\text{ rk}_{\bold e'}(f(s)) = f(s)$.
\end{enumerate}
\end{enumerate}
\end{definition}

\begin{observation}
\label{m4.9d}    1) If $\bold p$ is a strict 
$\iota$-system \und{then} $\bold p$ is a weak $\iota$-system.

\noindent
2) In Definition \ref{m4.6}, from $(j)\boxplus \bullet_6$ recalling
$(j)\oplus \bullet_1 + \bullet_4$ we can deduce: rk$_{\bold
   e}(f(s))= f(s)$ for $s \in Y_{\bold d}$.

\noindent
3) In $\oplus \bullet_5$ of (j) of \ref{m4.6} \wilog \, rk$_{\bold
   e}(y) > \zeta + 7$ as we can use $g + 7$.
\end{observation}

\begin{definition}
\label{m4.9} 
1) We say that a weak $\iota$-system $\bold p$ is weakly normal \und{when}:
\medskip

\noindent
\begin{enumerate}
\item[$\bullet_1$]   in $(d)(\beta)$ of Definition 
\ref{m4.3} we add $\bold e \in \Sigma(\bold d) 
\Rightarrow D_{\bold d} \subseteq D_{\bold e}$
\smallskip

\noindent
\item[$\bullet_2$]   in $(d)'(\beta)$ of Definition \ref{m4.4} we
add: if $(\bold e,h) \in \Sigma(\bold d)$ then $(\forall A \in
D_{\bold d})(h^{-1}(A) \in D_{\bold e})$.
\end{enumerate}
\medskip

\noindent
2) We say $\bold p$ is normal \when \, it is weakly normal and
\medskip

\noindent
\begin{enumerate}
\item[$\bullet_3$]  in Definition \ref{m4.3}, if $A \in D^+_{\bold d},
\bold d \in \bbD,f \in {}^{Y[\bold d]}\text{Ord}$ and $\zeta = \text{
rk}_{\bold d}(f)$ \then \, for some $\bold e \in \Sigma(\bold d)$ we have
$A \in D_{\bold e}$ and rk$_{\bold e}(f) = \zeta$
\sn
\item[$\bullet_4$]  in Definition \ref{m4.4}, if $\bold d \in \bbD,
f \in {}^{Y[\bold d]}\text{Ord}$, rk$_{\bold d}(f) =\zeta$ and
$A \in D^+_{\bold d}$, then for some $(\bold e,h) \in \Sigma(\bold d)$ we have
$\{s \in Y_{\bold e}:h(s) \in A\} \in D_{\bold e}$ and rk$_{\bold e}(f
\circ h) = \zeta$.
\end{enumerate}
\mn
3) We say $\bold p$ is semi normal when it is weakly normal and we
have $\bullet'_3,\bullet'_4$ holds where they are as above just ending
with $\ge \zeta$.
\end{definition}

\begin{claim}
\label{m4.10} 
Assume $\bold p$ is a weak $\iota$-system and $\bold d \in \bbD_{\bold p}$.

\noindent
0) If $\bold p,\bold q$ are as in \ref{m4.4d}, \then \, $\bold p$ is
   [weakly] normal iff $\bold q$ is.

\noindent
1) If $f,g \in {}^{Y[\bold d]}\text{\rm Ord}$ and $f <_{D_{\bold d}} g$
\und{then} {\rm rk}$_{\bold d}(f) < \text{\rm rk}_{\bold d}(g)$.

\noindent
2) If $f \in {}^{Y[\bold d]} \text{\rm Ord}$ and 
{\rm rk}$_{\bold d}(f) > 0$ \und{then} $\{s \in Y_{\bold d}:f(s) 
>0\} \in D^+_{\bold d}$.

\noindent
2A) If in addition $\bold p$ is semi-normal then
$\{s \in Y_{\bold d}:f(s)= 0\} = \emptyset$ {\rm mod} $D_{\bold d}$.

\noindent
3) {\rm rk}$_{\bold d}(f)$ depends just on $f/D_{\bold d}$ (and 
$\bold d$ and, of course, $\bold p$).
\end{claim}

\begin{PROOF}{\ref{m4.10}}  
0) Easy; note that by this part, below \wilog \, $\iota_{\bold p} = 2$.

\noindent
1) Let $f_1 \in {}^{Y[\bold d]}\text{Ord}$ be defined by
$f_1(s) = f(s)+1$.   So clearly $f_1 \le_{D_{\bold d}} g$ hence by
clause (i) of \ref{m4.3} (equivalently \ref{m4.4})
 we have rk$_{\bold d}(f_1) \le \text{ rk}_{\bold d}(g)$.  
Also $f_1 = f+1$ mod $D_{\bold d}$ hence by clause
(h) of \ref{m4.3} (equivalently \ref{m4.4}) we have 
rk$_{\bold d}(f_1) = \text{ rk}_{\bold d}(f) +1$.  
The last two sentences together give the desired conclusion.

\noindent
2) Toward contradiction assume the conclusion fails.  Let $f' \in
{}^{Y_{\bold d}}$Ord be constantly zero, so $f=f'$ mod $D_{\bold d}$ 
hence by part (3) we have rk$_{\bold d}(f') = \text{ rk}_{\bold d}(f) > 0$.  
By clause $(g)'$ of Definition \ref{m4.4}, the ``no hole" applied to
$(f',\bold d)$, there is a triple $(\bold e,h,g)$ as 
there, so $B := \{s:s \in Y_{\bold e}$ and
 $g(s) < f(h(s))\} \in D_{\bold e}$, i.e. $\{s \in Y_{\bold e}:g(s) <
   0\} \in D_{\bold e}$, contradiction.

Hence by the weak normality of
$\bold p$ we have $\{h(s):s \in B\} \ne
\emptyset$ mod $D_{\bold d}$ but this set includes 
$\{s \in Y_{\bold d}:f(s) > 0\}$.

\noindent
2A) Let $A = \{s \in Y_{\bold a}:f(s) = 0\}$, so toward contradiction
assume $A \in D^+_{\bold d}$.  As $\bold p$ is semi-normal we can find
a pair $(\bold e,h) \in \Sigma(\bold d)$ as in $\bullet'_4$ of
\ref{m4.9}(3) so $A_1 = \{t \in Y_{\bold e}:h_1(t) \in A\} \in
D_{\bold e}$ and rk$_{\bold e_1}(f \circ h) \ge \text{ rk}_{\bold
d}(f) > 0$, but clearly $\{t \in Y_{\bold e}:(f \circ h)(t) =0\} \in
D_{\bold e}$, contradiction by part (2).

\noindent
3) Use clause (i) of Definition \ref{m4.3} twice.
\end{PROOF}

\begin{theorem}
\label{m4.12}  {[\rm ZF]}  
Assume that $\bold p = (\bbD$,{\rm rk},$\Sigma,\bold j,\mu)$
is a strict rank 1-system (see Main Definition \ref{m4.3}) or just a
 strict 2-system.  \und{Then}
for every ordinal $\zeta$ there is $i < \text{\rm cf}(\mu)$
such that: if $\bold d \in \Bbb D_{\ge i}$ \und{then}
{\rm rk}$_{\bold d}(\zeta) = \zeta$, i.e. {\rm rk}$_{\bold d}(\langle
\zeta:s \in Y_{\bold d}\rangle) = \zeta$.
\end{theorem}

\begin{PROOF}{\ref{m4.12}}  
We shall use the notation:  
\medskip

\noindent
\begin{enumerate}
\item[$\odot_0$]  If there is an $i$ as required in the theorem 
for the ordinal $\zeta$ \then \, let $\bold i(\zeta)$ be the minimal such $i$
(otherwise, $\bold i(\zeta)$ is not well defined).
\end{enumerate}

Without loss of generality,
\medskip

\noindent
\begin{enumerate}
\item[$\odot_1$]    every $\bold d \in \bbD_{\bold p}$ is 
$(\bold p,\le (\text{cf}(\mu)))$-complete, i.e.
clause (f) of \ref{m4.3} for $\sigma_* :=
\text{ cf}(\mu)^+$ holds for every $\bold d \in \bbD$.
\end{enumerate}
\medskip

\noindent
[Why?  Let $i_*$ be the $i < \text{ cf}(\mu)$ which exists by clause
(f) of Definition \ref{m4.3}, \ref{m4.4} for $\sigma_*$.  Now we 
just replace $\bbD$ by $\bbD_{\ge i_*}$
(and $\bold j$ by $\bold j \rest \bbD_{\ge i_*}$, etc).] 

Clearly we have
\medskip

\noindent
\begin{enumerate}
\item[$\odot_2$]   rk$_{\bold d}(\zeta) \ge \zeta$ for $\zeta$ an
ordinal and $\bold d \in \bbD$.
\end{enumerate}
\medskip

\noindent
[Why?  We can prove this by induction on $\zeta$ for all $\bold d \in
\bbD$, by clauses (h) + (i) of Definition \ref{m4.3}.]

As a warmup we shall note that:
\medskip

\noindent
\begin{enumerate}
\item[$\odot_3$]  if $\bold d \in \bbD$ and $\zeta < \sigma_*$ or
just $\bold d \in \bbD$ and is hereditarily $(\bold p,\le
\zeta)$-complete which means that 
every $\bold e$ in the $\Sigma$-closure of $\{\bold
d\}$ is $(\bold p,\le \zeta)$-complete \und{then}:
\begin{enumerate}
\item[$(\alpha)$]  rk$_{\bold d}(\zeta) = \zeta$
\smallskip

\noindent
\item[$(\beta)$]  $f \in {}^{Y[\bold d]} \zeta \Rightarrow 
\text{ rk}_{\bold d}(f) < \zeta$.
\end{enumerate}
\end{enumerate}
\medskip

\noindent
[Why?  Note that as $\zeta < \sigma_*$ clearly $\bold d$ is
$(\bold p,\le\zeta)$-complete by $\odot_1$ and clause (f) of
\ref{m4.3}, so we can assume that $\bold d$ is hereditarily 
$(\bold p,\le \zeta)$-complete.
We prove the statement inside $\odot_3$ by induction on the ordinal
$\zeta$ (for all hereditarily $(\bold p,\le \zeta)$-complete
$\bold d \in \bbD$).  Note that for $\varepsilon
< \zeta,``\bold d$ is $(\bold p,\le \zeta)$-complete" implies ``$\bold d$ 
is $(\bold p,\le \varepsilon)$-complete", we shall use this freely.      

Arriving to $\zeta$, to prove clause $(\beta)$ 
let $f \in {}^{Y[\bold d]} \zeta$ and for
$\varepsilon < \zeta$ we define
$A_\varepsilon := \{t \in Y_{\bold d}:f(t) = \varepsilon\}$, so
$\langle A_\varepsilon:\varepsilon < \zeta\rangle$ is a well 
defined partition of $Y_{\bold d}$ so the sequence exists,  
hence as ``$\bold d$ is hereditarily $(\bold p,\le \zeta)$-complete"
recalling $(*)$ from clause $(f)'$ of \ref{m4.4} for 
some triple $(\bold e,h,\varepsilon)$ we have 
$(\bold e,h) \in \Sigma(\bold d)$ and $\varepsilon < \zeta$ and
$h^{-1}(A_\varepsilon) \in D_{\bold e}$ and rk$_{\bold e}(f \circ
h) = \text{ rk}_{\bold d}(f)$.

Now $f \circ h = \langle \varepsilon:t \in Y_{\bold e}\rangle$ mod $D_{\bold
e}$ hence by Claim \ref{m4.10}(3) we have rk$_{\bold e}(f \circ h) 
= \text{ rk}_{\bold e}(\varepsilon)$.  But the assumptions on $\bold
d$ holds for $\bold e$ hence by the induction
hypothesis on $\zeta$ we know that rk$_{\bold e}(\varepsilon) =
\varepsilon$ and $\varepsilon < \zeta$ so together rk$_{\bold d}(f
\circ h) < \zeta$, so clause $(\beta)$ of $\odot_3$ holds.  

To prove clause $(\alpha)$ first
consider $\zeta = 0$; if rk$_{\bold d}(\zeta) > 0$ by clause (g) of
Definition \ref{m4.3}, \ref{m4.4} there
are $(\bold e,h) \in \Sigma(\bold d)$ and $g \in 
{}^{Y[\bold e]}\text{Ord}$ 
such that $g < \langle \zeta:t \in Y_{\bold e}\rangle$
mod $D_{\bold e}$, so for some $t \in Y_{\bold e}$ we have $g(t) <
\zeta$ but $\zeta = 0$, contradiction; this is close to \ref{m4.10}(2).

Second, consider $\zeta > 0$, so by $\odot_2$ we have rk$_{\bold
d}(\zeta) \ge \zeta$ and assume toward contradiction that rk$_{\bold
d}(\zeta) > \zeta$, so by clause (g) of Definition \ref{m4.3}, \ref{m4.4}
there is a triple $(\bold e,h,g)$ as
there.  Now apply clause $(\beta)$ of $\odot_3$ for $\zeta$ (which we
have already proved) recalling $(\bold e,h) \in \Sigma(\bold d)$
so also $\bold e$ is $(\bold p,\le \zeta)$-complete.  We 
get rk$_{\bold e}(g) < \zeta$, a contradiction.  So $\odot_3$ indeed holds.]

Now as in the desired equality we have already proved one inequality
in $\odot_2$, we need to prove only the other inequality.  We do 
it by induction on $\zeta$.
\bigskip

\noindent
\und{Case 1}:  $\zeta < \mu$.

By clause (f) for some $i < \text{ cf}(\mu)$ we have $\bold d \in 
\bbD \wedge \bold j(\bold d) \ge i \Rightarrow \bold d$ is $(\bold p,\le
\zeta)$-complete, hence by $\odot_3(\alpha)$ we have rk$_{\bold
d}(\zeta) = \zeta$, as required.
\bigskip

\noindent
\und{Case 2}:  $\zeta = \xi +1$.

By clause (h) of Definition \ref{m4.3} we have $\bold d \in \bbD
\Rightarrow \text{ rk}_{\bold d}(\zeta) = \text{ rk}_{\bold d}(\xi)
+1$.  Hence $\bold d \in \bbD_{\ge \bold i(\xi)} \Rightarrow
\text{ rk}_{\bold d}(\zeta) = \text{ rk}_{\bold d}(\xi) +1= \xi +1 =
\zeta$, so $\bold i(\xi)$ exemplifies that $\bold i(\zeta)$
exists and is $\le \bold i(\xi)$ so we are done.
\bigskip

\noindent
\und{Case 3}:  $\zeta$ is a limit ordinal $\ge \mu$ of cofinality $\ne
\text{ cf}(\mu)$.  

So for each $\xi < \zeta$ by the
induction hypothesis $\bold i(\xi) < \text{ cf}(\mu)$ is well
defined.  For $i < \text{ cf}(\mu)$ let $u_i := \{\xi < \zeta:\bold
i(\xi) \le i\}$, so is well defined; moreover the sequence
$\langle u_i:i < \text{ cf}(\mu)\rangle$ exists 
and is $\subseteq$-increasing.  If $i <
\text{ cf}(\mu) \Rightarrow \sup(u_i) < \zeta$ then $\langle
\sup(u_i):i < \text{ cf}(\mu)\rangle$ is a $\le$-increasing sequence
of ordinals $< \zeta$ with limit $\zeta$.  So as 
cf$(\zeta) \ne \text{ cf}(\mu)$ necessarily for some $i_* < \text{
cf}(\mu)$ the set $S := \{\xi:\xi < \zeta$ and $\bold i(\xi) <
i_*\}$ is an unbounded subset of $\zeta$.  We shall prove that $\bold
i(\zeta)$ is well defined and $\le i_*$.
\bigskip

\noindent
\und{Subcase 3A}:  cf$(\zeta) \ge \mu$.

Let $\bold d \in \bbD_{\ge i_*}$ and $g \in {}^{Y[\bold d]}
\zeta$ be given.  Clearly Rang$(g)$ is a subset of $\zeta$ of cardinality $<
\theta(Y_{\bold d})$ which by clause $(b)(\alpha)$ of \ref{m4.3} is 
$< \mu \le \text{ cf}(\zeta)$ hence we
can fix $\xi \in S$ such that Rang$(g) \subseteq \xi$, hence by clause
(i) of \ref{m4.3}, rk$_{\bold d}(g) \le \text{ rk}_{\bold d}(\xi)$
but $\bold i(\xi) = i_*$ and $\bold d \in \bbD_{\ge i_*}$ hence
rk$_{\bold d}(\xi) = \xi < \zeta$ so together rk$_{\bold d}(g) <
\zeta$.  As this holds for every $\bold d \in \bbD_{\ge i_*}$
by the no-hole clause $(g)'$ and clause (e)$(\gamma)$ of \ref{m4.4} it follows 
that $\bold d \in \bbD_{\ge i_*} \Rightarrow 
\text{ rk}_{\bold d}(\zeta) \le \zeta$ as required.
\bigskip

\noindent
\und{Subcase 3B}:  cf$(\zeta) < \mu$ (but still cf$(\zeta) \ne 
\text{ cf}(\mu))$.

Let $\langle \zeta_\varepsilon:\varepsilon < \text{ cf}(\zeta)\rangle$
be an increasing sequence of ordinals from $S$ with limit $\zeta$.
Now let $j_* < \text{cf}(\mu)$ be such that $\bold d \in 
\Bbb D_{\ge j_*} \Rightarrow \bold d$ is 
$(\bold p,\text{cf}(\zeta)^+)$-complete, see
clause (f) of Definition \ref{m4.3}.

Now assume $\bold d \in \bbD_{\ge\text{max}\{i_*,j_*\}}$ and $g \in
{}^{Y[\bold d]} \zeta$.  For $\varepsilon < \text{ cf}(\zeta)$ let
$A_\varepsilon = \{t \in Y_{\bold d}:g(t) < \zeta_\varepsilon$ but
$j< \varepsilon \Rightarrow g(t) \ge \zeta_j\}$ so $\langle
A_\varepsilon:\varepsilon < \text{ cf}(\zeta)\rangle$ is well defined and 
is a partition of $Y_{\bold d}$.  Hence by clause (f) of Definition
\ref{m4.4} for
some $\varepsilon < \text{ cf}(\zeta)$ and $(\bold e,h) \in \Sigma(\bold
d)$ we have $h^{-1}(A_\varepsilon) \in D_{\bold e}$ and rk$_{\bold d}(g) =
\text{ rk}_{\bold e}(g \circ h)$; but $\bold j(\bold e)  \ge \bold
j(\bold d) \ge i_*,j_*$ and by the choice of $A_\varepsilon$ and
clause (i) of \ref{m4.3} the latter is $\le 
\text{ rk}_{\bold e}(\zeta_\varepsilon)$ hence as 
$\bold i(\zeta_\varepsilon) \le i_*$ the latter is $=
\zeta_\varepsilon < \zeta$.  As this holds for every $\bold d \in 
\bbD_{\ge \text{max}\{i_*,j_*\}}$
and $g \in {}^{Y[\bold d]}\zeta$, by the no-hole clause $(g)'$ of \ref{m4.4}
necessarily rk$_{\bold d}(\zeta) \le \zeta$.  
So max$\{i_*,j_*\} < \text{ cf}(\mu)$ is as required, so we are done.
\bigskip

\noindent
\und{Case 4}:  $\zeta \ge \mu$ is a limit ordinal
such that cf$(\zeta) = \text{ cf}(\mu)$.

Let $\langle \zeta_i:i < \text{ cf}(\zeta)\rangle$ be increasing with
limit $\zeta$.  Assume toward contradiction that for every $i <
\text{cf}(\mu)$ there is $\bold d_i \in \bbD_{\ge i}$ such 
that rk$_{\bold d_i}(\zeta) > \zeta$ \und{but} we do not assume that
such a sequence 
$\langle \bold d_i:i < \text{ cf}(\mu)\rangle$ exists.  Choose such
$\bold d_0$; as rk$_{\bold d_0}(\zeta) > \zeta$, 
clearly there are $f_0 \in {}^{Y[\bold d'_0]} \zeta$ and a member
$\bold d'_0$ of $\Sigma(\bold d_0)$, though not necessarily $Y_{\bold
d'_0} = Y_{\bold d}$,  such that
\medskip

\noindent
\begin{enumerate}
\item[$\odot_4$]    rk$_{\bold d'_0}(f_0) = \zeta$ 
\end{enumerate}
\medskip

\noindent
[Why?  By using clause $(g)'$ of \ref{m4.4}.]

Note
\medskip

\noindent
\begin{enumerate}
\item[$\odot_5$]  $\bold i(f_0(t))$ is well defined for every 
$t \in Y_{\bold d'_0}$.
\end{enumerate}
[Why holds?  Because $f_0(t) < \zeta$ and the induction hypothesis.]

For $j_1 < \text{ cf}(\zeta),j_0 < \text{ cf}(\mu)$ 
let $A_{j_1,j_0} = \{t \in Y_{\bold d'_0}:f_0(t) < \zeta_{j_1}$ and
$(\forall j < j_1)(f_0(t) \ge \zeta_j)$ and 
$\bold i(f_0(t)) = j_0\}$.  By clause $(f)'$ of
\ref{m4.4} applied to the pair $(\bold d'_0,f_0)$ and the partition
$\langle A_{j_1,j_0}:j_1 < \text{ cf}(\zeta),j_0 < \text{ cf}
(\mu)\rangle$, for some $(\bold d_*,h_*) \in \Sigma(\bold d'_0)$ 
we have rk$_{\bold d_*}(f_0 \circ h_*) = \zeta$ and
for some $j_1,j_2$ we have $h^{-1}_*(A_{j_1,j_0}) \in D_{\bold d_*}$.  
By \ref{m4.10}(3) for some $f = f_0 \circ h_*$ mod $D_{\bold d_*}$ and
letting $\bold d := \bold d_*$ we have
\medskip

\noindent
\begin{enumerate}
\item[$\odot_6$]    $(a) \quad \bold d \in \bbD$
\smallskip

\noindent
\item[${{}}$]   $(b) \quad f \in {}^{Y[\bold d]}\text{Ord}$
\smallskip

\noindent
\item[${{}}$]   $(c) \quad$ rk$_{\bold d}(f) = \zeta$
\smallskip

\noindent
\item[${{}}$]   $(d) \quad t \in Y_{\bold d} \Rightarrow \bold i(f(t)) = j_2
\wedge f(t) < \zeta_{j_1} \wedge (\forall j < j_1)(f(t)  \ge \zeta_j)$.
\end{enumerate}
\mn
Next
\mn
\begin{enumerate}
\item[$\odot_7$]   letting $\xi := \zeta_{j_1}$, 
clause $\boxplus$ from \ref{m4.6} for $(\bold d,\zeta,\xi,f,j_0)$.
\end{enumerate}
\mn
[Why?  We check the six demands
\mn
\begin{enumerate}
\item[$\bullet_1$]  $``f \in {}^{Y[\bold d]}\xi$" 
which holds by $\odot_6(b) + (d)$
\smallskip

\noindent
\item[$\bullet_2$]  ``rk$_{\bold d}(f) = \zeta$" 
which holds by $\odot_6(c)$
\smallskip

\noindent
\item[$\bullet_3$]  $``\xi < \zeta$" which holds as 
$(\forall i < \text{ cf}(\mu))(\zeta_i < \zeta)$
\smallskip

\noindent
\item[$\bullet_4$]  ``cf$(\zeta) = \text{ cf}(\mu)$" 
which holds by the case assumption
\smallskip

\noindent
\item[$\bullet_5$]   $j_2 < \text{ cf}(\mu)$ obvious
\smallskip

\noindent
\item[$\bullet_6$]  $s \in Y_{\bold d} \wedge \bold e' \in \bbD_{\ge j_0}
\Rightarrow \text{ rk}_{\bold e'}(f(s)) = f(s)$ holds by $\odot_6(d)$.
\end{enumerate}
\mn
So $\odot_7$ indeed holds.]

Now by $\odot_7$, clause (j) of Definition \ref{m4.6}(1) 
applied with $\bold d,\zeta,\xi = \zeta_{j_1},f,j_0$ here standing for
$\bold d,\zeta,\xi,f,j_0$ there, we can find $j$ as there.  Let
$i_2 = \text{ max}\{j,j_1,j_0,\bold i(\zeta_{j_1})\}$ so $i_2 < \text{
cf}(\mu)$ and choose $\bold e_0 \in \bbD_{\ge i_2}$ such that
rk$_{\bold e_0}(\zeta) > \zeta$ as in the beginning of the case.
As rk$_{\bold e_0}(\zeta) > \zeta$ by clause $(g)'$ of \ref{m4.4}
there are $\bold e_1 \in \Sigma(\bold e_0)$ and 
$g \in {}^{Y[\bold e_1]}\zeta$ such that rk$_{\bold e_1}(g) \ge \zeta$ so 
$g < \langle \zeta:t \in Y_{\bold e_1}\rangle$.  Now \wilog
\medskip

\noindent
\begin{enumerate}
\item[$\odot_8$]  $(a) \quad$ rk$_{\bold e_1}(g) = \zeta +1$
\smallskip

\noindent
\item[${{}}$]  $(b) \quad \zeta_* = \sup\{g(t):t \in Y_{\bold e_1}\} <
\zeta$
\sn
\item[${{}}$]  $(c) \quad \bold j(\bold e_1) \ge i_2$.
\end{enumerate}
\medskip

\noindent
[Why?  Because we can use $g' \in {}^{Y[\bold e_1]}\zeta$
defined by $g'(t) = g(t) +2$ for $t \in Y_{\bold e_1}$, by clause (h) of
\ref{m4.3}, rk$_{\bold e_1}(g') = \text{ rk}_{\bold e_1}(g) + 2 > \zeta$.
By clause $(e)(\gamma)$
we have $\bold j(\bold e_1) \ge \bold j(\bold e_0) \ge i_2$.  
Now we find $(\bold d''_2,h'') \in \Sigma(\bold e_1)$ 
and $g_2$ as in the proof of $\odot_6$ and rename.] 

Also without loss of generality
\medskip

\noindent
\begin{enumerate}
\item[$\odot_9$]    $t \in Y_{\bold e_1} \Rightarrow g(t) \ge \zeta_{j_1}$.
\end{enumerate}
\medskip

\noindent
[Why?  Let $A_0 = \{t \in Y_{\bold e_1}:g(t) < \zeta_{j_1}\},A_1 = \{t
\in Y_{\bold e_1}:g(t) \ge \zeta_{j_1}\}$ so by clause $(f)'$ of
\ref{m4.4} for some pair
$(\bold e_2,h) \in \Sigma(\bold e_1)$ we have
rk$_{\bold e_2}(g \circ h) = \text{ rk}_{\bold e_1}(g) = \zeta +1$ and 
$(h^{-1}(A_0) \in D_{\bold e_2}) \vee (h^{-1}(A_1) \in D_{\bold e_2})$.  
So if $h^{-1}(A_0) \in D_{\bold e_2}$ then by clause (i) of \ref{m4.3}, 
rk$_{\bold e_1}(g \circ h) \le \text{ rk}_{\bold e_2}(\xi)$ 
but $\bold i(\xi)$ is well defined $\le i_2
\le \bold j(\bold e_1) \le \bold j(\bold e_2)$ 
so rk$_{\bold e_2}(\xi) = \xi$
together rk$_{\bold e_2}(g \circ h) \le \xi$ contradicting the
previous sentence.  Hence $h^{-1}(A_0) \notin D_{\bold e_2}$ 
so $h^{-1}(A_1) \in D_{\bold e_2}$.  
Let $g' \in {}^{Y[\bold e_2]}\text{Ord}$ be defined by 
$g'(t)$ is $(g \circ h)(t)$ if $t \in h^{-1}(A_1)$ and is $\zeta_{j_1} + 1$ 
if $t \in h^{-1}(A_0)$.  By Claim \ref{m4.10}(3) we have
rk$_{\bold e_2}(g') = \text{ rk}_{\bold e_2}(g \circ h)$ so $(\bold
e_2,g')$ satisfies all requirements on the pair $(\bold e_1,g)$ and $t
\in Y_{\bold e_2} \Rightarrow g'(t) \ge \zeta_{j_1} > 0$, 
so we have justified the non-loss of generality.]

Recall $\xi := \zeta_{j_1}$ and let $\bold e = \bold e_1$.
By the choice of $j$ after $\odot_6$, 
i.e. as in clause (j) of \ref{m4.6}, recalling $\bold e \in \bbD_{\ge
j}$ we shall get a contradiction to the choice of 
$(\bold d,\xi,\zeta,f,j_0,\bold e,g,j)$.  To justify it we have to
recall by $\odot_7$ that the quintuple $(\bold d,\zeta,\xi,f,j_0)$ 
satisfies $\boxplus$ of \ref{m4.6}(j) and then we 
prove that the triple $(\bold e,g,j)$ satisfies $\oplus$ of \ref{m4.6}(j).

Now $\oplus$ of \ref{m4.6} says:
\medskip

\noindent
\begin{enumerate}
\item[$\bullet_1$]  $``\bold e \in \bbD_{\ge j}"$ as 

as $j \ge i_2,\bold e_0 \in \bbD_{\ge i_2}$ and $\bold e = 
\bold e_1 \in \Sigma(\bold e_0)$
\smallskip

\noindent
\item[$\bullet_2$]  $``g \in {}^{Y[\bold e]}\zeta$" 

which holds as $g \in {}^{Y[\bold e]}\zeta$
\smallskip

\noindent
\item[$\bullet_3$]  $``g(t) \in [\xi,\zeta_*)$" 

holds as $g(t) < \zeta$ by $\bullet_2 + \odot_8(b)$ and 
$g(t) \ge \zeta_{j_1} = \xi$ by $\odot_9$
\smallskip

\noindent
\item[$\bullet_4$]  $``j \ge j_0"$

holds as $j \ge i_2 \ge j_0$
\smallskip

\noindent
\item[$\bullet_5$]  ``rk$_{\bold e}(g) > \zeta$" 

holds by $\odot_8(a)$.
\end{enumerate}
\medskip

\noindent
So we really get a contradiction.  
\end{PROOF}

\begin{definition}
\label{m4.14}  
1) We say that the pair $(\bold d,\bold e)$ commute 
(or 6-commute) for $\bold p$ \when \, $\bold d,\bold e \in 
\bbD_{\bold p}$ and {\rm rk}$_{\bold d}(f) \ge \text{\rm rk}_{\bold
e}(g)$ whenever $(f,g,\bar f,\bar g)$ is a 
$(\bold p,\bold d,\bold e)$-rectangle, see below; fixing $f,g$ we may
say $(\bold d,\bold e)$ commute for $f,g$.

\noindent
2) We say that $(\bold d,\bold e,f,g,\bar f,\bar g)$ 
is $\bold p$-rectangle or $(f,g,\bar f,\bar g)$ is a
$(\bold p,\bold d,\bold e)$-rectangle \und{when}:
\medskip

\noindent
\begin{enumerate}
\item[$\circledast$]   $(a) \quad \bold d \in \bbD_{\bold p}$
\smallskip

\noindent
\item[${{}}$]   $(b) \quad \bold e \in \bbD_{\bold p}$
\smallskip

\noindent
\item[${{}}$]   $(c) \quad \bar g = \langle g_t:t \in 
Y_{\bold e}\rangle$ and $g_t \in {}^{Y[\bold d]}\text{\rm Ord}$ for
$t \in Y_{\bold e}$
\smallskip

\noindent
\item[${{}}$]  $(d) \quad g \in {}^{Y[\bold e]}\text{\rm Ord}$ is defined by
$g(t) = \text{\rm rk}_{\bold d}(g_t)$
\smallskip

\noindent
\item[${{}}$]  $(e) \quad f_s \in {}^{Y[\bold e]}\text{\rm Ord}$ is
defined by $f_s(t) = g_t(s)$
\smallskip

\noindent
\item[${{}}$]  $(f) \quad \bar f = \langle f_s:s \in Y[\bold d]\rangle$
\smallskip

\noindent
\item[${{}}$]   $(g) \quad f \in {}^{Y[\bold d]}\text{\rm Ord}$ is defined
by $f(s) = \text{\rm rk}_{\bold e}(f_s)$.
\end{enumerate}
\end{definition}

\begin{claim}
\label{m4.16}   [Assume {\rm ZF + AC}$_{< \mu}$]
If $\bold p = (\bbD,\text{\rm rk},\Sigma,\bold i,\mu)$ be a weak 
rank 1-system \then \, $\bold p$ is a strict rank 1-system when there
is a function $\Sigma_1$ such that (and we may say $\Sigma_1$ witness it):
\medskip

\noindent
\begin{enumerate}
\item[$(*)_0$]   $\Sigma_1$ a function with domain $\bbD$
\sn
\item[$(*)_1$]  $\Sigma_1(\bold d) \subseteq \Sigma(\bold d)$ for
$\bold d \in \bbD$
\smallskip

\noindent
\item[$(*)_2$]   for every $\bold d,\zeta,\xi,f,j_0$
satisfying $\boxplus$ of \ref{m4.6}, for some $j <
\text{\rm cf}(\mu)$ for every $\bold e \in \bbD_{\ge j}$ we have
\begin{enumerate}
\item[$(a)$]   $\bold e$ is $(\bold p,\le \Sigma_1(\bold d))$-complete
\smallskip

\noindent
\item[$(b)$]  if $\bold d_* \in \Sigma_1(\bold d),\bold e_* \in
\Sigma_1(\bold e)$ then $(\bold d_*,\bold e_*)$ commute (for $\bold p$)
see \ref{m4.14}, at least for $f \in {}^{Y[\bold d]}\zeta$ and any $g \in
{}^{Y[\bold e]}[\xi,\zeta)$
\end{enumerate}
\item[$(*)_3$]   we strengthen clause (g) of Definition \ref{m4.3} to
\begin{enumerate}
\item[$(g)^+$]  add: {\rm rk}$_{\bold e}(f) = 
\text{\rm rk}_{\bold d}(f)$ and $\bold e \in \Sigma_1(\bold d)$
\end{enumerate}  
\item[$(*)_4$]  {\rm AC}$_{Y[\bold d]}$ and {\rm AC}$_{\Sigma_1(\bold d)}$ 
whenever $\bold d \in \bbD$.
\end{enumerate}
\end{claim}

\begin{remark}  1) In $(*)_1$, can we make $j$ depend on $f$ and 
a partition of $Y_{\bold d}$?  Will be somewhat better.

\noindent
2) We can similarly prove this for a weak rank 2-system.  It is
   natural though not necessary to add $(\bold e,h) \in \Sigma_1(\bold
   d) \Rightarrow Y_{\bold e} = Y_{\bold d} \wedge h = 
\text{ id}_{Y_{\bold d}}$.
\end{remark}

\begin{PROOF}{\ref{m4.16}}   
Let $\bold d,\zeta,\xi,f,j_0$ satisfying $\boxplus$ of
\ref{m4.6}(j) be given and we should find $j < \text{ cf}(\mu)$ such
that for no pair $(\bold e,g)$ clause $\oplus$ there holds.  Without
loss of generality $s \in Y_{\bold d} \Rightarrow f(s) > 0$.

Let $j < \text{ cf}(\mu)$ be as in $(*)_2$ in the claim and \wilog \, $j
> j_0$ and we shall prove that $j$ is as required in clause (j) of Definition 
\ref{m4.6}, this is enough.  So assume $\bold e \in \bbD_{\ge j},g
\in {}^{Y[\bold e]}[\xi,\zeta]$ and toward contradiction, 
$(j,\zeta,\xi,\bold e,g)$ satisfy $\oplus$ there.  For each $t \in
Y_{\bold e}$ clearly $g(t) < \zeta = \text{ rk}_{\bold d}(f)$ hence
by clause $(g)^+$ of $(*)_3$, see (g) of Definition \ref{m4.3}, ``no hole",
there are $g_t \in {}^{Y[\bold d]}\xi$ and $\bold d_t \in
\Sigma_1(\bold d)$ such
that $g_t <_{D_{\bold d_t}} f$ and rk$_{\bold d_t}(g_t) = g(t)$,
without loss of generality $g_t < \text{ max}(f,1_{Y[\bold d]}) =f$ and by
 the $(*)_3$, ``we add" also rk$_{\bold d_t}(f) = \text{ rk}_{\bold d}(f)$.

As AC$_{Y_\bold e}$ by $(*)_4$, we can choose such sequence
$\langle(g_t,\bold d_t):t \in Y_{\bold e}\rangle$.  Now $\bold e$ is
$(\bold p,\le \Sigma_1(\bold d))$-complete and $(\bold d,\bold e)$
commute for $\bold p$, by clauses (a),(b) respectively 
of $(*)_2$ (i.e. by the choice of
$j$ and as $\bold e \in \bbD_{\ge j}$), hence we can find
$\bold e_* \in \Sigma_1(\bold e)$ and $\bold d_* \in \Sigma_1(\bold
d)$ such that rk$_{\bold e_*}(g) = \text{ rk}_{\bold e}(g) = \zeta$
and $\{t \in Y_{\bold e}:\bold d_t = \bold d_*\}$ belongs to
$D_{\bold e_*}$.  For $s \in Y_{\bold d} = Y_{\bold d_*}$ let $f_s \in
{}^{Y[\bold e_*]}\text{Ord}$ be defined by $f_s(t) = g_t(s)$ so
$f_s(t) = g_t(s) < \xi$ and let $f' \in {}^{Y[\bold d_*]}\text{Ord}$ 
be defined by $f'(s) = \text{ rk}_{\bold e_*}(f_s)$ and 
let $\bar f = \langle f_s:s \in Y_{\bold d_*}\rangle$.

Fixing $s \in Y_{\bold d_*}$ we have $t \in Y_{\bold e_*} \Rightarrow
f_s(t) = g_t(s) < \text{ Max}\{f(s),1\} = f(s)$, i.e.
 $f_s < \langle f(s):t \in Y_{\bold e_*}\rangle$ hence
rk$_{\bold e_*}(f_s) < \text{ rk}_{\bold e_*}(f(s))$.  Now by
$\boxplus \bullet_6$ from \ref{m4.6}, as $j_0 \le j \le \bold
j(\bold e_*)$ we have $s \in Y_{\bold e_*} \Rightarrow \text{
rk}_{\bold e_*}(f(s)) = f(s)$ so $s \in Y_{\bold e_*} \Rightarrow
\text{ rk}_{\bold e_*}(f_s) < f(s)$, i.e. $f' < f$.

Clearly $(f',g,\bar f,\bar g)$ is a $(\bold p,\bold d_*,\bold
e_*)$-rectangle hence by clause (b) of $(*)_2$ of the assumptions,
i.e. the choice of $(\bold e,g)$ and Definition \ref{m4.14}(2) we know that
rk$_{\bold d_*}(f') \ge \text{ rk}_{\bold e_*}(g)$.

But recall that rk$_{\bold e_*}(g) = \text{ rk}_{\bold e}(g)$ by the
choice of $\bold e_*$.  We get a contradiction by
\medskip

\noindent
\begin{enumerate}
\item[$(*)$]   $\zeta = \text{ rk}_{\bold d}(f) = \text{ rk}_{\bold d_*}(f) >
\text{ rk}_{\bold d_*}(f') \ge \text{ rk}_{\bold e_*}(g) =
\text{ rk}_{\bold e}(g) \ge \zeta$.
\end{enumerate}
\medskip

\noindent
[Why those inequalities?  By $\bullet_2$ of $\boxplus$ from
\ref{m4.6} we are assuming; as $\bold d_* \in \{\bold
d_t:t \in Y_{\bold e}\}$ and the choice of the $\bold d_t$'s; as $f'
<_{D_{\bold d_*}} f$ and \ref{m4.10}(3); by an inequality above; by
the choice of $\bold e_*$; by $\bullet_5$ of $\oplus$ of \ref{m4.6}.]
\end{PROOF}

\section {Finding Systems} 
\bigskip

\noindent
\S (4A) \underline{Building weak rank systems and measurable}

\begin{claim}
\label{e5.3}  [{\rm ZF + DC}]

If $\circledast_1$ holds and $\bold p_{\kappa,\theta} =\bold p_{\bar \kappa} = 
\bold p = (\bbD,\text{\rm rk},\Sigma,\bold j,\mu)$ is defined in 
$\circledast_2$ \und{then} $\bold p$ is a weak rank 1-system, even semi
normal (and $(g)^+$ of \ref{m4.16} holds) where:
\medskip

\noindent
\begin{enumerate}
\item[$\circledast_1$]  $(a) \quad \bar \kappa = \langle \kappa_i:i <
\partial\rangle$ is an increasing sequence of regular cardinals

\hskip25pt $>\partial = \text{\rm cf}(\partial)$ 
with limit $\mu$ such that if $i < \partial$ is a limit ordinal then

\hskip25pt $\kappa_i = (\Sigma\{\kappa_j:j < i\})^+$
\smallskip

\noindent
\item[${{}}$]  $(b) \quad \theta^*$ is a cardinal or $\infty$
\smallskip

\noindent
\item[$\circledast_2$]   $(a) \quad \bbD_i = \{J:J$ is a
$\kappa_i$-complete ideal on some $\kappa = \kappa_J < \mu$ 
including 

\hskip40pt $[\kappa]^{<\kappa}$ and satisfying {\rm cf}$(J,\le) <
\theta^*$

\hskip40pt (and if $\theta^* = \infty$ we stipulate this as the empty 

\hskip40pt demand) such that $\beta < \kappa \Rightarrow \{\beta\} \in J\}$
and let $\bbD = \Bbb D_0$
\smallskip

\noindent
\item[${{}}$]  $(b) \quad$ if $\bold d = J \in \bbD_i,J$ an 
ideal on $\kappa_J := \cup\{A:A \in J\}$

\hskip25pt  then we let $Y_{\bold d} = \kappa_J$ and $D_{\bold d}$ be the
filter dual to the ideal $J$
\smallskip

\noindent
\item[${{}}$]   $(c) \quad \bold j(J) = \text{\rm min}\{i:J$ is not
$\kappa^+_{i+1}$-complete$\}$ 
\smallskip

\noindent
\item[${{}}$]  $(d) \quad \Sigma(J) = 
\{J + B:B \supseteq A$ and $\kappa_J \backslash B$ is not in $J\}$
\smallskip

\noindent
\item[${{}}$]   $(e) \quad$ {\rm rk}$_J(f)$ is
as in Definition \ref{z0.21}.
\end{enumerate}
\end{claim}

\begin{PROOF}{\ref{e5.3}} 
   So we have to check all the clauses in
Definition \ref{m4.3}.
\medskip

\noindent
\und{Clause (a)}:   As $\mu = \Sigma\{\kappa_i:i < \partial\}$, the
sequence $\langle \kappa_i:i < \partial\rangle$ is increasing and
$\kappa_0 > \partial$ (all by $\circledast_1$)
clearly $\mu$ is a singular cardinal (and $\partial = \text{
cf}(\mu)$).
\medskip

\noindent
\und{Clause (b)}:  Let $\bold d \in \bbD$, so $\bold d = J$.
\medskip

\noindent
\und{Subclause $(\alpha)$}:  So $Y_{\bold d} = \kappa_J < \mu$ hence
$\theta(Y_{\bold d}) = \theta(\kappa_J) = \kappa^+_J < \mu$ recalling
$\mu$ is a limit cardinal and the definition of $\bbD = \bbD_0$ in
clause (a) of $\circledast_2$.
\medskip

\noindent
\und{Subclause $(\beta)$}:  Also obvious.
\medskip

\noindent
\und{Clause (c)}:  For $f \in {}^{(\kappa_{\bold d})}\text{Ord}$, rk$_{\bold
d}(f)$ as defined in $\circledast_2(e)$, is an ordinal recalling Claim
\ref{z0.23}(1).
\medskip

\noindent
\und{Clause (d)$(\alpha)$}:  

Trivial.
\medskip

\noindent
\und{Clause (d)$(\beta)$}:

Trivially $\bold e \in \Sigma(\bold d) \Rightarrow Y_{\bold e} =
Y_{\bold d} \wedge D_{\bold e} \supseteq D_{\bold d}$; 
so ``$\bold p$ is weakly normal", see Definition \ref{m4.9}, moreover
``$\bold p$ is semi-normal" as rk$_D(f) \le \text{ rk}_{D+A}(f)$ for $A
\in D^+$.
\medskip

\noindent
\und{Clause (e)}:  

Obvious from the definitions.
\medskip

\noindent
\und{Clause (f)}:  

Let $\sigma < \mu$ be given and choose $i < \partial$ such that
$\sigma < \kappa_i$.
Let $\bold d \in \bbD$ be such that $j = \bold j(\bold d) \ge i$
hence $D = D_{\bold d}$ is a filter on some $\kappa_J$, so assume
$\cup\{A_\varepsilon:\varepsilon < \varepsilon^*\} = \kappa_J$ 
and $\varepsilon^* < \kappa_i$.  Now $D$ is $\kappa_i$-complete and  
(see \ref{z0.25}(2))
we have rk$_{D_{\bold d}}(f) = \text{ min}\{\text{rk}_{D+A_\varepsilon}(f):
\varepsilon < \varepsilon^*$ and
$A_\varepsilon \in D_{\bold d}^+\}$ which is what is needed as $A_\varepsilon
\in D^+_{\bold d} \Rightarrow \bold d + (\kappa_j \backslash
A_\varepsilon) \in \Sigma(\bold d)$.
\medskip

\noindent
\und{Clause (g)}:  By \ref{z0.23}(2).

Moreover, the stronger version with $\bold e = \bold d$ holds so in
particular $(g)^+$ of \ref{m4.16} holds.
\medskip

\noindent
\und{Clause (h)}:

Easy.  On the one hand, as $g < f$, the definition of 
rk$_{\bold d}(f)$, we have rk$_{\bold d}(f) \ge \text{ rk}_{\bold
d}(g)+1$.   On the other hand, if $g'<f$ mod $D_{\bold d}$ 
then $g' \le g$ mod $D$ hence by clause (i) below we have
rk$_{\bold d}(g') \le \text{ rk}_{\bold d} (g) 
< \text{\rm rk}_{\bold d}(g)+1$, as this holds for every $g' < f$ mod 
$D_{\bold d}$ we have rk$_{\bold d}(f) \le \text{ rk}_{\bold d}(g)+1$.  
Together we are done.
\medskip

\noindent
\und{Clause (i)}:

Obvious.  
\end{PROOF}

\begin{discussion}  
Assume $\mu$ is a singular cardinal, $\mu = \sum\limits_{i < \kappa}
\mu_i,\kappa = \text{ cf}(\kappa) < \mu_i < \mu$ and $\mu_i$ is 
increasing with $i$.  Assume that for each $i$ there is a pair
$(D,Y),D$ is a $\mu_i$-complete ultra-filter on $Y,\theta(Y) <
\mu$. This seems to be a good case, but either we have ``$D$ is a
$(\le \theta(Y))$-complete" so ${}^{Y}\text{Ord}/D$ is ``dull" or
$\theta(Y) > \kappa =$ completeness$(D)$ and so there is a
$\kappa$-complete non-principal ultrafilter on $\kappa$ and
on $\kappa < \mu$ so $\mu = \sup$(measurables $\cap \mu$).
\end{discussion}

\begin{claim}
\label{e5.13} {\rm [ZF + DC + AC$_{< \mu}$]}

Assume $\mu$ is singular and $\mu = \sup(\mu \cap \text{ the class of
measurable cardinals})$, (equivalently for every $\kappa < \mu$ there is a
$\kappa$-complete non-principal ultrafilter on some $\kappa' < \mu$).
Let $\bar\kappa = \langle \kappa_i:i < \text{\rm cf}(\mu)\rangle$ be
increasing with limit $\mu,\kappa_i > \text{\rm cf}(\mu)$ such that
for $i$ limit $\kappa_i = (\Sigma\{\kappa_j:j < i\})^+$ and $\kappa_i$
is measurable for $i$ non-limit.

\und{Then} $\bold p = \bold p^{\text{\rm uf}}_{\bar\kappa}$ is a strict
rank 1-system \und{where} $\bold p$ is defined by
\medskip

\noindent
\begin{enumerate}
\item[$\circledast$]   $(a) \quad \bbD_{\ge i} = \{J$: dual $(J)$ is a
non-principal ultra-filter which is

\hskip25pt $\kappa_i$-complete on some $\kappa = \kappa_J < \mu\}$

\hskip25pt so naturally $Y_J = \kappa_J$ and $D_J =$ {\rm dual}$(J)$
\smallskip

\noindent
\item[${{}}$]   $(b) \quad \bold j(J) = \text{\rm min}\{i:J$ is not
$\kappa^+_{i+1}$-complete$\}$, well defined
\smallskip

\noindent
\item[${{}}$]  $(c) \quad \Sigma(J) = \{J\}$
\smallskip

\noindent
\item[${{}}$]   $(d) \quad \text{\rm rk}_J(f) = 
\text{\rm rk}_{\text{\rm dual}(J)}(f)$ as in \ref{z0.21}.
\end{enumerate}
\end{claim}

\begin{PROOF}{\ref{e5.13}}  
We can check clauses (a)-(i) of \ref{m4.3} as in the
proof of \ref{e5.3}.

We still have to prove the ``strict", i.e. we should prove clause (j) from
Definition \ref{m4.6}.  We prove this using Claim \ref{m4.16}, we
choose $\Sigma_1(\bold d) := \{\bold d\} \subseteq \Sigma(\bold d)$ for
$\bold d \in \bbD_{\bold p}$ so it suffices to prove $(*)_0 - (*)_4$
of \ref{m4.16}.

So in Claim \ref{m4.16}, we have $(*)_0,(*)_1$ hold by the choice of
$\Sigma_1$, and concerning $(*)_3$ in \ref{e5.3} we prove $(g)^+$, and
$(*)_4$ holds as for each $\kappa < \mu$ we have AC$_\kappa$ as $\kappa
< \mu$ by an assumption and for $\bold d \in \bbD$ we have
AC$_{\Sigma_1(\bold d)}$, as $\Sigma_1(\bold d)$ is a singleton.

Note that
\mn
\begin{enumerate}
\item[$\boxplus_1$]  if $\kappa < \mu < \theta(\cP(\kappa))$ then $\mu$
is not measurable.
\end{enumerate}
\mn
Now we are left with proving $(*)_2$, so
let $\bold d \in \bbD_{\bold p},\zeta,\xi,f \in 
{}^{Y[\bold d]}\zeta$ be given as in (j) of $\boxplus$ in \ref{m4.6}, and we
should find $j$ as there.  

Let $j < \partial = \text{ cf}(\mu)$ be such
that $\theta({\cP}(\kappa)) < \kappa_j$, and let $\bold e 
\in \bbD_{\ge j}$.  Now clause (a) is trivial as $|\Sigma_1(\bold
d)|=1$, and clause (b) says that `` the pair $(\bold d,\bold e)$
commute", see Definition \ref{m4.14} recalling $\Sigma_1(\bold d) =
\{\bold d\}$.  So let $(f,g,\bar f,\bar g)$
be a $(\bold p,\bold d,\bold e)$-rectangle, see Definition
\ref{m4.14}(2), and we should prove that rk$_{\bold e}(g) \le \text{
rk}_{\bold d}(f)$; let $Y_1 = Y_{\bold e},Y_2 = Y_{\bold d}$.

To prove this we apply \ref{c3.3} or \ref{c3.14}, but the $f,\bar
f$ are interchanged with $g,\bar g$; we check $\oplus(a)-(g)$ from
\ref{c3.3}.  They hold by $\circledast(a)-(f)$ of Definition
\ref{m4.14}.

Concerning $\boxplus(a),(b)$ from \ref{c3.3}, 
``AC$_{Y_1}$"  holds as AC$_{< \mu}$ holds
and the definition of $\bold p$.  Lastly, we should prove
$\boxplus(a)$ there which says ``$D_{\bold d}$ does 2-commute 
with $D_{\bold e}$" which holds by Case 2 of Claim \ref{c3.8}.
\end{PROOF}

\begin{conclusion}
\label{5e.17}
[AC$_{< \mu},\mu$ a singular cardinal]
Assume $\mu = \sup\{\lambda < \mu:\lambda$ is a measurable
cardinal$\}$.  \underline{Then} for every ordinal $\zeta$ for some
$\kappa < \lambda$ we have rk$_D(\zeta) = \zeta$ for every
$\kappa$-complete ultrafilter on some cardinality $< \mu$.
\end{conclusion}

\begin{PROOF}  It suffices to prove this for the case $\mu$ has
cofinality $\aleph_0$.  Now we can apply Claim \ref{e5.13} and Theorem
\ref{m4.12}. 
\end{PROOF}

\section {pseudo true cofinality}

We repeat here \cite[\S5]{Sh:938}.

\centerline {Pseudo PCF} 
\medskip

We try to develop pcf theory with little choice.  We deal only with
$\aleph_1$-complete filters, and replace cofinality and other basic
notions by pseudo ones, see below.
This is quite reasonable as with choice there is no difference.

This section main result are \ref{r9}, existence of filters with
pseudo-true-cofinality; \ref{r14}, giving a parallel 
of $J_{< \lambda}[\alpha]$.

In the main case we may (in addition to ZF) assume 
DC + AC$_{{\cP}({\cP}(Y))}$; this will be continued in \cite{Sh:955}.

\begin{hypothesis}
\label{r1}  ZF 
\end{hypothesis}

\begin{definition}
\label{r2}  
1) We say that a partial order $P$ is $(< \kappa)$-directed \when \, every
   subset $A$ of $P$ of power $< \kappa$ has a common upper bound.

\noindent
1A) Similarly $P$ is $(\le S)$-directed.

\noindent
2) We say that a partial order $P$ is
pseudo $(< \kappa)$-directed \und{when} it is $(< \kappa)$-directed
and moreover 
every subset $\cup\{P_\alpha:\alpha < \delta\}$ has a common upper 
bound \und{when}:
\medskip

\noindent
\begin{enumerate}
\item[$(a)$]    if $\delta < \kappa$ is a limit ordinal
\smallskip

\noindent
\item[$(b)$]   $\bar P = \langle P_\alpha:\alpha <  \delta\rangle$
is a sequence of non-empty subsets of $P$
\smallskip

\noindent
\item[$(c)$]   if $\alpha_1 < \alpha_2,p_1 \in P_{\alpha_1}$ and
$p_2 \in P_{\alpha_2}$ then $p_1 <_P p_2$.
\end{enumerate}
\mn
2A) For a set $S$ we say that the partial order $P$ is
pseudo $(\le S)$-directed \when \, $\cup\{P_s:s \in S\}$ has a common
upper bound whenever
\medskip

\noindent
\begin{enumerate}
\item[$(a)$]    $\langle P_s:s \in S\rangle$ is a sequence
\sn
\item[$(b)$]    $P_s \subseteq P$
\sn
\item[$(c)$]    if $s \in S$ then $P_s$ has a common upper bound. 
\end{enumerate}
\end{definition}

\begin{definition}
\label{r3}   We say that a partial (or quasi) order $P$ has
pseudo true cofinality $\delta$ \und{when}: $\delta$ is a limit ordinal
and there is a sequence $\langle
P_\alpha:\alpha < \delta\rangle$ such that 
\medskip

\noindent
\begin{enumerate}
\item[$(a)$]    $P_\alpha \subseteq P$ and
$\delta = \sup\{\alpha < \delta:P_\alpha$ non-empty$\}$
\smallskip

\noindent
\item[$(b)$]    if $\alpha_1 < \alpha_2 < \delta,p_1 \in
P_{\alpha_1},p_2 \in P_{\alpha_2}$ then $p_1 <_P p_2$
\smallskip

\noindent
\item[$(c)$]   if $p \in P$ then for some $\alpha < \delta$ and $q
\in P_\alpha$ we have $p \le_P q$.
\end{enumerate}
\end{definition}

\begin{remark}
\label{r4}   
0) See \ref{r2}(2) and \ref{r8}(1).

\noindent
1) We could replace $\delta$ by a partial order $Q$.

\noindent
2) The most interesting case is in Definition \ref{r6}.

\noindent
3) We may in Definition \ref{r3} demand $\delta$ is a regular
   cardinal.

\noindent
4) Usually in clause (a) of Definition \ref{r3} \wilog \, 
$\bigwedge\limits_{\alpha} P_\alpha
   \ne \emptyset$, as \wilog \, $\delta = \text{cf}(\delta)$ using
   $P'_\alpha = P_{f(\alpha)}$ where $f(\alpha) =$ the $\alpha$-th
member of $C$ where $C$ is an unbound subset of 
$\{\beta < \delta:P_\beta \ne \emptyset\}$ of order type cf$(\delta)$.  
Why do we allow
   $P_\alpha = \emptyset$? as it is more natural in \ref{r12}(1), but
   can usually ignore it.
\end{remark}

\begin{example}
\label{r5}  
Suppose we have a limit ordinal $\delta$ and a sequence $\langle
A_\alpha:\alpha < \delta\rangle$ of sets with $\prod \limits_{\alpha
< \delta} A_\alpha = \emptyset$; moreover $u \subseteq \delta =
\sup(u) \Rightarrow \prod \limits_{\alpha \in u} A_\alpha =
\emptyset$.  Define a partial order $P$ by:
\medskip

\noindent
\begin{enumerate}
\item[$(a)$]   its set of elements is 
$\{(\alpha,a):a \in A_\alpha$ and $\alpha < \delta\}$
\smallskip

\noindent
\item[$(b)$]   the order is $(\alpha_1,a_1) <_P (\alpha_2,a_2)$ iff 
$\alpha_1 < \alpha_2$ (and $a_\ell \in
A_{\alpha_\ell}$ for $\ell=1,2$).
\end{enumerate}
It seems very reasonable to say that $P$ has true cofinality but there
is no increasing cofinal sequence.
\end{example}

\begin{definition}
\label{r6}  
1) For a set $Y$ and sequence $\bar \alpha = \langle \alpha_t:
t \in Y\rangle$ of ordinals and cardinal $\kappa$ we define

\begin{equation*}
\begin{array}{clcr}
\text{ps-tcf-fil}_\kappa(\bar \alpha) = \{D:&D \text{ a } 
\kappa\text{-complete filter on } Y \text{ such that } (\Pi \bar\alpha/D) \\
  &\text{ has a pseudo true cofinality}\};
\end{array}
\end{equation*}

see below.

\noindent
2) We say that $\Pi \bar \alpha/D$ or $(\Pi \bar\alpha,D)$ or $(\Pi
\bar \alpha,<_D)$ has pseudo true 
cofinality $\gamma$ \und{when} $D$ is a filter on $Y = 
\text{ Dom}(\bar \alpha)$ and $\gamma$ is a
limit ordinal and the partial order $(\Pi\bar \alpha,<_D)$ essentially
does\footnote{so necessarily $\{s \in Y:\alpha_s > 0\}$ belongs to
$D$ but is not necessarily empty; if it is $\ne Y$ then $\Pi \bar\alpha =
\emptyset$, so pedantically this is wrong, $(\Pi \bar \alpha,<_D)$
does not have any pseudo true cofinality hence we say ``essentially"
but usually we shall ignore this
\und{or} assume $\bigwedge\limits_{t} \alpha_t \ne 0$ when not said
otherwise.}, i.e., 
there is a sequence $\bar{\cF}= \langle {\cF}_\beta:\beta < \gamma
\rangle$ satisfying:
\medskip

\noindent
\begin{enumerate}
\item[$\circledast_{\bar{\cF}}$]  $(a) \quad {\cF}_\beta 
\subseteq \{f \in {}^Y\text{Ord}:f <_D \bar \alpha\}$
\smallskip

\noindent
\item[${{}}$]   $(b) \quad {\cF}_\beta \ne 0$
\smallskip

\noindent
\item[${{}}$]   $(c) \quad$ if 
$\beta_1 < \beta_2,f_1 \in {\cF}_{\beta_1}$ and
$f_2 \in {\cF}_{\beta_2}$ then $f_1 < f_2$ mod $D$
\smallskip

\noindent
\item[${{}}$]  $(d) \quad$ if $f \in {}^Y{\text{\rm Ord}}$ and $f < 
\bar\alpha$ mod $D$ then for some $\beta <
\gamma$ we have $g \in {\cF}_\beta \Rightarrow$

\hskip25pt  $f < g$ mod $D$ (by clause
(c) this is equivalent to: for some $\beta < \gamma$ 

\hskip25pt and some $g \in {\cF}_\beta$ we have $f \le g$ mod $D$).
\end{enumerate}
\medskip

\noindent
3) ps-pcf$_\kappa(\bar \alpha) = \text{
ps-pcf}_{\kappa\text{\rm -comp}}(\bar\alpha) := \{\gamma$: there is a 
$\kappa$-complete
filter $D$ on $Y$ such that $\Pi \bar \alpha/D$ has pseudo true cofinality
$\gamma$ and $\gamma$ is minimal for $D\}$.

\noindent
4) pcf-fil$_{\kappa,\gamma}(\bar\alpha) = \{D:D$ a $\kappa$-complete filter on
$Y$ such that $\Pi \bar \alpha/D$ has true cofinality $\gamma\}$.

\noindent
5) In part (2) if $\gamma$ is minimal we call it ps-tcf$(\Pi
\bar\alpha,D)$ or simply
ps-tcf$(\Pi \bar \alpha,<_D)$; note that it is a well 
defined (regular cardinal).
\end{definition}

\begin{claim}
\label{r7}  
1) If $\lambda =$ \,{\rm ps-tcf}$(\Pi\bar\alpha,<_D)$, \und{then} 
$(\Pi \bar \alpha,<_D)$ is pseudo $(< \lambda)$-directed.

\noindent
1A) If $\theta(S) < \lambda = \text{\rm ps-tcf}(\Pi \bar\alpha,<_D)$
\then \, $(\Pi \bar \alpha,<_D)$ is pseudo $(\le S)$-directed.

\noindent
2) Similarly for any quasi order.

\noindent
3) If {\rm cf}$(\alpha_t) \ge \lambda = \text{\rm cf}(\lambda)$ for $t \in Y$
\then \, $(\Pi \bar \alpha,<_D)$ is $\lambda$-directed.

\noindent
4) Assume {\rm AC}$_\alpha$ for $\alpha < \lambda$.  If {\rm cf}$(\alpha_s) \ge
 \lambda$ for $s \in Y$ \then \, $(\Pi \bar\alpha,<_D)$ is pseudo
 $\lambda$-directed. 
\end{claim}

\begin{PROOF}{\ref{r7}}  1), 1A), 2)  As in \ref{r8}(1) below.  

\noindent
3) So assume $\cF \subseteq \Pi\bar\alpha$ satisfies $|\cF| <
   \lambda$, so there is a sequence $\langle f_\alpha:\alpha <
   \mu\rangle$ listing $\cF$ for some $\mu < \lambda$.  Let $f \in
\Pi\bar\alpha$ be defined by $f(s) = \sup\{f_\alpha(s):\alpha <
   \mu\}$, now $f(s) < \alpha(s)$ as cf$(\alpha_s) \ge \lambda > \mu$.

\noindent
4) So assume $\bar P = \langle P_\alpha:\alpha < \delta\rangle,\delta$
   a limit ordinal $< \lambda$ and $P_\alpha \subseteq \Pi\bar\alpha$
   non-empty and $\alpha < \beta < \delta  \wedge f \in P_\alpha \wedge
   g \in P_\beta \Rightarrow f <_D g$.  As AC$_\delta$ holds we can
   find a sequence $\bar f = \langle f_\alpha:\alpha \in \delta\rangle
   \in \prod\limits_{\alpha < \beta} P_\alpha$ and apply part (3).
\end{PROOF}

\begin{claim}
\label{r8}
Let $\bar \alpha = \langle \alpha_s:s \in Y \rangle$ and $D$ is a
filter on $Y$.

\noindent  
0) If $\Pi \bar\alpha/D$ has pseudo true
cofinality \und{then} {\rm ps-tcf}$(\Pi \bar \alpha,<_D)$ is a regular
cardinal; similarly for any partial order.

\noindent
1) If $\Pi \bar \alpha/D$ has pseudo true
cofinality $\gamma_1$ and true cofinality $\gamma_2$ \und{then} 
{\rm cf}$(\gamma_1) = \text{\rm cf}(\gamma_2) = \text{\rm ps-tcf}(\Pi
\bar \alpha,<_D)$, similarly for any  partial order.

\noindent
2) {\rm ps-pcf}$_\kappa(\bar \alpha)$ 
is a set of regular cardinals so if
$\Pi \bar \alpha/D$ has pseudo true cofinality \then \,
{\rm ps-tcf}$(\Pi \bar \alpha,<_D)$ is $\gamma$ where $\gamma =
\text{\rm cf}(\gamma)$ and $\Pi\bar\alpha/D$ has pseudo cofinality
$\gamma$.

\noindent
3) Always {\rm ps-pcf}$_\kappa(\bar\alpha)$ has cardinality $<
   \theta(\{D:D$ a $\kappa$-complete filter on $Y\})$.

\noindent
4) If $\bar\beta = \langle \beta_s:s \in Y\rangle \in {}^Y\text{\rm Ord}$
   and $\{s:\beta_s = \alpha_s\} \in D$ \then \, {\rm ps-tcf}$(\Pi
   \bar\alpha/D) = \text{\rm ps-tcf}(\Pi \bar\beta/D)$ so one is well
   defined iff the other is.
\end{claim}

\begin{PROOF}{\ref{r8}}  0) By the definitions.

\noindent
1) Let $\langle {\cF}^\ell_\beta:\beta <
\gamma_\ell\rangle$ exemplify ``$\Pi \bar \alpha/D$ has pseudo true
cofinality $\gamma_\ell$" for $\ell=1,2$.  Now
\medskip

\noindent
\begin{enumerate}
\item[$(*)$]    if $\ell \in \{1,2\}$ and $\beta_\ell < \gamma_\ell$
then for some $\beta_{3 - \ell} < \gamma_{3 - \ell}$ we have $g_1 \in
{\cF}^\ell_{\beta_\ell} \wedge g_2 \in 
{\cF}^{3-\ell}_{\beta_{3-\ell}} \Rightarrow g_1 <_D g_2$.
\end{enumerate}
\medskip

\noindent
[Why?  Choose $g^\ell \in {\cF}^\ell_{\beta_\ell +1}$, choose
$\beta_{3 - \ell} < \gamma_{3-\ell}$ and $g_{3 -\ell} \in 
{\cF}^{3 -\ell}_{\beta_{3-\ell}}$ such that $g^\ell < g^{3-\ell}$ mod
$D$.  Clearly $f \in \cF^\ell_{\beta_\ell} \Rightarrow f <_D g^\ell
<_D g^{3-\ell}$ so $g^{3-\ell}$ is as required.]
\newline

Hence
\begin{enumerate}
\item[$(*)$]   $h_1:\gamma_1 \rightarrow \gamma_2$ is well defined
when

$h_1(\beta_1) = \text{ Min}\{\beta_2 < \gamma_2:(\forall g_1 \in 
{\cF}^1_{\beta_1})(\forall g_2 \in \cF^2_{\beta_2})(g_1 < g_2 
\text{ mod } D)\}$.
\end{enumerate}
\mn
Clearly $h$ is non-decreasing and it is not eventually constant (as
$\cup\{{\cF}^1_\beta:\beta < \gamma_1\}$ is cofinal in $\Pi 
\bar\alpha/D$) and has range unbounded in $\gamma_2$ (similarly).
\newline

The rest should be clear.

\noindent
2) Follows.

\noindent
3),4)  Easy.    
\end{PROOF}

Concerning \cite{Sh:835}
\begin{claim}
\label{r9}
\und{The Existence of true cofinality filter}   
[$\kappa > \aleph_0 + \text{\rm DC} + 
\text{\rm AC}_{< \kappa}$]  If
\medskip

\noindent
\begin{enumerate}
\item[$(a)$]    $D$ is a $\kappa$-complete filter on $Y$
\smallskip

\noindent
\item[$(b)$]    $\bar \alpha \in {}^Y\text{\rm Ord}$
\smallskip

\noindent
\item[$(c)$]     $\delta := \text{\rm rk}_D(\bar \alpha)$ satisfies
{\rm cf}$(\delta) \ge \theta(\text{\rm Fil}^1_\kappa(Y))$,
see below.
\end{enumerate}
\und{Then} for some $D'$ we have
\medskip

\noindent
\begin{enumerate}
\item[$(\alpha)$]  $D'$ is a $\kappa$-complete filter on $Y$
\smallskip

\noindent
\item[$(\beta)$]   $D' \supseteq D$
\smallskip

\noindent
\item[$(\gamma)$]  $\Pi \bar \alpha/D'$ has pseudo true 
cofinality, in fact, {\rm ps-tcf}$(\Pi \bar \alpha,<_{D'}) = 
\, \text{\rm cf(rk}_D(\bar \alpha))$.
\end{enumerate}
\end{claim}

Recall from \cite{Sh:835}
\begin{definition}
\label{r9a}
0) Fil$^1_\kappa(Y) = \{D:D$ a $\kappa$-complete filter on $Y\}$ and
   if $D \in \text{ Fil}^1_\kappa(Y)$ then Fil$^1_\kappa(D) = \{D' \in
   \text{ Fil}^1_\kappa(Y):D \subseteq D'\}$.

\noindent
1)  Fil$^4_\kappa(Y) = \{(D_1,D_2):D_1 \subseteq D_2$
are $\kappa$-complete filters on $Y\}$.

\noindent
2) $J[f,D]$ where $D$ is a filter on $Y$ and $f \in {}^Y\text{Ord}$ is
   $\{A \subseteq Y:A = \emptyset$ mod $D$ or rk$_{D+A}(f) > \text{
   rk}_D(f)\}$. 
\end{definition}

\begin{remark}
\label{r9b}  1) On the Definition of 
pseudo $(< \kappa,1 + \gamma)$-complete $D$ see \ref{z0.51}; 
we may consider changing the definition of
Fil$^1_\kappa(Y)$ to $D$ is $\aleph_1$-complete and pseudo$(<
\kappa,1+\gamma))$-complete filter on $Y$.
\end{remark}

\begin{PROOF}{\ref{r9}}
\underline{Proof of the Claim \ref{r9}}
   
Recall $\{y \in Y:\alpha_y = 0\} = \emptyset$ mod $D$ as  
rk$_D(\langle \alpha_y:y \in Y\rangle) = \delta > 0$ but $f_1,f_2 \in
{}^Y \text{ Ord} \wedge (f_1=f_2 \text{ mod } D) \Rightarrow \text{
rk}_D(f_1) = \text{ rk}_D(f_2)$ hence without loss of generality 
$y \in Y \Rightarrow \alpha_y > 0$.

Let $\bbD = \{D':D'$ is a filter on $Y$ extending $D$ which is
$\kappa$-complete$\}$.  So $\theta(\bbD) \le 
\theta(\text{Fil}^1_{\aleph_1}(Y)) \le \text{ cf}(\delta)$.  
For any $\gamma < \text{ rk}_D(\bar \alpha)$ 
and $D' \in \bbD$ let
\medskip

\noindent
\begin{enumerate}
\item[$(*)_2$]   $(a) \quad {\cF}_{\gamma,D'} = \{f \in \Pi \bar
\alpha:\text{rk}_D(f) =\gamma$ and $D'$ is dual$(J[f,D])\}$
\smallskip

\noindent
\item[${{}}$]   $(b) \quad {\cF}_{D'} = \cup\{{\cF}_{\gamma,D'}:
\gamma < \text{ rk}_D(\bar\alpha)\}$
\smallskip

\noindent
\item[${{}}$]  $(c) \quad \Xi_{\bar \alpha,D'} 
= \{\gamma < \text{ rk}_D(\bar \alpha):{\cF}_{\gamma,D'} \ne
\emptyset\}$
\smallskip

\noindent
\item[${{}}$]  $(d) \quad {\cF}_\gamma = \cup\{{\cF}_{\gamma,D''}:
D'' \in \bbD\}$.
\end{enumerate}

Now
\medskip

\noindent
\begin{enumerate}
\item[$(*)_3$]  if $\gamma < \text{ rk}_D(\bar \alpha)$ 
then ${\cF}_\gamma \ne \emptyset$.
\end{enumerate}
\medskip

\noindent
[Why?  By \ref{z0.23}(2) there is $g \in {}^Y\text{Ord}$
such that $g < \bar\alpha$ mod $D$ and rk$_D(g) = \gamma$ and  \wilog \, $g \in
\Pi \bar\alpha$.  Now let $D' = \text{ dual}(J[g,D])$, so $(D,D') \in
\text{ Fil}^4_\kappa(Y)$ by \ref{z0.29}(1) (using AC$_{< \kappa}$) the
fitler $D'$ is $\kappa$-complete so 
$D' \in \bbD$ and clearly $g \in \cF_{\gamma,D'}$, see
\ref{z0.23}(2), but $\cF_{\gamma,D'} \subseteq \cF_\gamma$ so
$\cF_\gamma \ne 0$; here we use AC$_{< \kappa}$.]
\medskip

\noindent
\begin{enumerate}
\item[$(*)_4$]  $\{\sup(\Xi_{\bar\alpha,D'}):D' \in \bbD$ 
and $\Xi_{\bar\alpha,D'}$ is bounded in rk$_D(\bar \alpha)\}$ is 
a subset of rk$_{D'}(\bar \alpha)$
which has cardinality $< \theta(\bbD) \le
\theta(\text{Fil}^1_\kappa(Y)) \le \text{ cf}(\delta)$.
\end{enumerate}
\medskip

\noindent
[Why?  The function $D' \mapsto \sup(\Xi_{\bar\alpha,D'})$ witness this.]
\medskip

\noindent
\begin{enumerate}
\item[$(*)_5$]  the set in $(*)_4$ is bounded below rk$_D(\bar\alpha)$
so let $\gamma(*) < \text{\rm rk}_D(\bar\alpha)$ be its supremum.
\end{enumerate}
\medskip

\noindent
[Why?  By $(*)_4$.]
\medskip

\noindent
\begin{enumerate}
\item[$(*)_6$]  there is $D' \in \bbD$ such that $\Xi_{\bar\alpha,D'}$
is unbounded in $(\Pi \bar\alpha,<_{D'})$.
\end{enumerate}
\medskip

\noindent
[Why?  Choose $\gamma < \text{\rm rk}_D(\bar\alpha)$ such that 
$\gamma > \gamma(*)$.  By $(*)_3$ there is $f \in \cF_{\gamma(*)}$ and
by $(*)_2(d)$ for some $D' \in \bbD$ we have $f \in
\cF_{\gamma(*),D'}$ so by the choice of $\gamma(*)$ the set 
$\Xi_{\bar\alpha,D'}$ cannot be bounded in rk$_D(\bar\alpha)$.]
\medskip

\noindent
\begin{enumerate}
\item[$(*)_7$]   if $\gamma_1 < \gamma_2$ are from 
$\Xi_{\bar\alpha,D'}$ and $f_1 \in {\cF}_{\gamma_1,D'},f_2 \in 
{\cF}_{\gamma_2,D'}$ \und{then} $f_1 <_{D'} f_2$.
\end{enumerate}
\medskip

\noindent
[Why?  By \ref{z0.23}.]

Together we are done: by $(*)_6$ there is $D' \in \bbD$ such that 
$\Xi_{\bar\alpha,D'}$ is unbounded in rk$_D(\bar \alpha)$.  Hence
$\bar{\cF} = \langle {\cF}_{\gamma,D'}:
\gamma \in \Xi_{\bar\alpha,D'}\rangle$ witness that $(\Pi \bar\alpha,<_{D'})$ 
has pseudo true cofinality by $(*)_7$, and so ps-tcf$(\Pi \bar\alpha,<_D) 
= \text{ cf(otp}(\Xi_{\bar\alpha,D'})) 
= \text{ cf}(\text{rk}_D(\bar \alpha))$, so 
we are done.  
\end{PROOF}

So we have 
\begin{dc}
\label{r10}  
1) We say that $\delta =$ ps-tcf$_{\bar D}(\bar \alpha)$, where 
$\delta$ is a limit ordinal \und{when}, for some set $Y$:
\medskip

\noindent
\begin{enumerate}
\item[$(a)$]  $\bar \alpha \in {}^Y\text{Ord}$
\item[$(b)$]   $\bar D = (D_1,D_2)$
\item[$(c)$]  $D_1 \subseteq D_2$ are $\aleph_1$-complete
filters on $Y$ 
\item[$(d)$]   rk$_{D_1}(\bar\alpha) = \delta = \sup(\Xi_{\bar
D,\bar\alpha})$ where $\Xi_{\bar D,\bar \alpha} = \{\gamma <
\text{ rk}_{D_1}(\bar \alpha)$: for some $f < \bar \alpha$ mod $D_1$, we
have rk$_{D_1}(f) =\gamma$ and $D_2 =$ dual$(J[f,D_1]\}$.
\end{enumerate}

\noindent
2) If $D_1$ is $\aleph_1$-complete filter on $Y,\bar \alpha =
\langle \alpha_t:t \in Y\rangle$ and cf$(\alpha_t) \ge
\theta(\text{Fil}^1_{\aleph_1}(Y))$ for $t \in Y$ \und{then} for some
$\aleph_1$-complete filter $D_2$ on $Y$ extending $D_1$ we have
ps-tcf$_{(D_1,D_2)}(\bar \alpha)$ is well defined.

\noindent
3) Moreover in part (2) there is a definition giving for any
$(Y,D_1,D_2,\bar \alpha)$ as there, a sequence $\langle 
{\cF}_\gamma:\gamma < \delta\rangle$ exemplifying the value of
ps-tcf$_{\bar D}(\bar \alpha)$.
\end{dc}

\begin{PROOF}{\ref{r10}}
2), 3)  Let $\delta := \text{ rk}_{D_1}(f)$, so by Claim \ref{r7}(3) we
have cf$(\delta) \ge \theta(\text{Fil}^1_{\aleph_1}(Y))$ hence by Claim
\ref{r9} above and its proof the conclusion holds: the proof is needed
for ``$\delta = \sup(\Xi_{\bar D,\alpha})"$, noting observation
\ref{r10d} below.
\end{PROOF}

\begin{observation}
\label{r10d}
1) [DC] or just [AC$_{\aleph_0}$].

Assume $D$ is an $\aleph_1$-complete filter on $Y$ and $f,f_n \in
{}^Y\text{\rm Ord}$ for $n < \omega$ and $f(t) = \sup\{f_n(t):n <
\omega\}$.  Then {\rm rk}$_D(f) = \sup\{\text{\rm rk}_D(f_n):n < \omega\}$.
\end{observation}

\begin{remark}  Similarly for other amounts of completeness, see \ref{r13}.
\end{remark}

\begin{PROOF}{\ref{r10d}}  As rk$_D(f) = \text{ min}\{\text{rk}_{D+A_n}(f):n <
\omega\}$ if $\cup\{A_n:n < \omega\} \in D,A_n \in D^+$ by \ref{z0.25}
or see \cite{Sh:71}.
\end{PROOF}

\begin{remark}
Also in \ref{z0.25}(2) can use AC$_Y$ only, i.e. omit
the assumption DC, a marginal point here.
\end{remark}

\begin{claim}
\label{r11}  [{\rm AC}$_{< \theta}$]  The ordinal
$\delta$ has cofinality $\ge \theta$ \when \,:
\medskip

\noindent
\begin{enumerate}
\item[$\circledast$]   $(a) \quad \delta = \text{\rm rk}_D(\bar\alpha)$
\item[${{}}$]  $(b) \quad \bar \alpha = \langle \alpha_y:y \in Y
\rangle \in {}^Y\text{\rm Ord}$
\item[${{}}$]  $(c) \quad D$ is an $\aleph_1$-complete filter on $Y$
\item[${{}}$]  $(d) \quad y \in Y \Rightarrow \text{\rm cf}
(\alpha_y) \ge \theta$.
\end{enumerate}
\end{claim}

\begin{PROOF}{\ref{r11}}  
Note that $y \in Y \Rightarrow \alpha_y > 0$.  Toward
contradiction assume cf$(\delta) < \theta$ so $\delta$ has a cofinal
subset $C$ of cardinality $< \theta$.  For each $\beta <
\delta$ for some $f \in {}^Y \text{Ord}$ we have rk$_D(f) = \beta$
and $f <_D \bar\alpha$ and \wilog \, $f \in \prod\limits_{y \in Y}
\alpha_y$.  By AC$_{< \theta}$ there is a sequence $\langle
f_\beta:\beta \in C\rangle$ such that $f_\beta \in \prod\limits_{y \in Y}
\alpha_y,f_\beta <_D \bar\alpha$ and rk$_D(f_\beta) = \beta$.  Define
$g \in \prod\limits_{y \in Y} \alpha_y$ by $g(y) =
\cup\{f_\beta(y):\beta \in C$ and $f_\beta(y) < \alpha_t\}$.  By
clause (d) we have $[y \in Y \Rightarrow g(y) < \alpha_y]$,
so $g <_D \bar\alpha$, hence rk$_D(\bar g) < \text{
rk}_D(\alpha)$ but by the choice of $g$ we have 
$\beta \in C \Rightarrow f_\beta \le_D g$ hence $\beta \in C \Rightarrow
\beta = \text{ rk}_D(f_\beta) \le \text{ rk}_D(g)$ hence $\delta =
\sup(C) \le \text{ rk}_D(g)$, contradiction.  
\end{PROOF}

\begin{observation}
\label{r12}
1) Assume $(\bar\alpha,D)$ satisfies 
\medskip

\noindent
\begin{enumerate}  
\item[$(a)$]   $D$ a filter on $Y$ and $\bar\alpha = \langle
\alpha_t:t \in Y\rangle$ and each $\alpha_t$ is a limit ordinal
\smallskip

\noindent
\item[$(b)$]   $\bar\cF = \langle {\cF}_\beta:\beta < \partial\rangle$
exemplify $\partial = \text{\rm ps-tcf}(\Pi \bar \alpha,<_D)$ \und{so}
we demand just $\partial = \sup\{\beta <
\partial:\cF_\beta \ne \emptyset\}$
\smallskip

\noindent
\item[$(c)$]  ${\cF}'_\beta = \{f \in \prod\limits_{t \in Y}
\alpha_t$: for some $g \in {\cF}_\beta$ we have $f=g$ {\rm mod} $D\}$.
\end{enumerate}

\und{Then}: $\langle {\cF}'_\beta:\beta < \partial \rangle$
exemplify $\partial = \text{\rm ps-tcf}(\Pi \bar \alpha,<_D)$ that is
\mn
\begin{enumerate}
\item[$(\alpha)$]   $\bigcup\limits_{\beta < \gamma} \cF'_\beta$ is
cofinal in $(\Pi \bar\alpha,<_D)$
\smallskip

\noindent
\item[$(\beta)$]  for every $\beta_1 < \beta_2 < \partial$ and $f_1 \in
{\cF}'_{\beta_1}$ and $f_2 \in {\cF}'_{\beta_2}$ we have
$f_1 \le f_2$.
\end{enumerate}
\mn
2) Similarly, if $D,\bar{\cF}$ satisfies clauses (a),(b) above and 
$D$ is $\aleph_1$-complete and $\partial = 
\text{\rm cf}(\partial) > \aleph_0$ \then \, we can ``correct"
$\bar{\cF}$ to make it $\aleph_0$-continuous that is $\langle
{\cF}''_\beta:\beta < \partial\rangle$ defined in $(c)_1 + (c)_2$
below satisfies $(\alpha) + (\beta)$
above and $(\gamma)$ below and so is $\aleph_0$-continuous, (see
below) where
\medskip

\noindent
\begin{enumerate}
\item[$(c)_1$]  if $\beta < \partial$ and {\rm cf}$(\beta) \ne \aleph_0$
then ${\cF}''_\beta = {\cF}'_\beta$
\smallskip

\noindent
\item[$(c)_2$]   if $\beta < \partial$ and {\rm cf}$(\beta) = \aleph_0$ then
${\cF}''_\beta = \{\sup\langle f_n:n < \omega\rangle$: for some
increasing sequence $\langle \beta_n:n < \omega\rangle$ with limit
$\beta$ we have $n <\omega \Rightarrow f_n \in {\cF}'_{\beta_n}\}$,
see below
\smallskip

\noindent
\item[$(\gamma)$]   if $\beta < \partial$ and {\rm cf}$(\beta) = \aleph_0$ and 
$f_1,f_2 \in {\cF}''_\beta$ then $f_1 = f_2$ mod $D$. 
\end{enumerate}
\mn
3) This applies to any increasing sequence $\langle \cF_\beta:\beta <
\delta\rangle,\cF_\beta \subseteq {}^Y\text{\rm Ord},\delta$ a limit ordinal.
\end{observation}

\begin{PROOF}{\ref{r12}}
  Straightforward.
\end{PROOF}

\begin{definition}
\label{r13} 
0) If $f_n \in {}^Y\text{Ord}$
for $n < \omega$, \und{then} sup$\langle f_n:n < \omega\rangle$ is
defined as the function $f$ with domain $Y$ such that
$f(t) = \cup\{f_n(t):n < \omega\}$.

\noindent
1) We say $\bar{\cF} = \langle {\cF}_\beta:\beta < \lambda\rangle$ 
exemplifying $\lambda = \text{ ps-tcf}(\Pi \bar\alpha,<_D)$ is 
weakly $\aleph_0$-continuous \und{when}:

if $\beta < \partial$, cf$(\beta) = \aleph_0$ and $f \in {\cF}_{\beta_n}$
then for some sequence $\langle(\beta_n,f_n):n<\omega\rangle$ we have
$\beta = \cup\{\beta_n:n < \omega\},\beta_n < \beta_{n+1} < \beta,f_n
\in {\cF}_{\beta_n}$ and $f = \sup\langle f_n:n <
\omega\rangle$; so if $D$ is $\aleph_1$-complete then $\{f/D:f \in
{\cF}_\beta\}$ is a singleton.

\noindent
2) We say it is $\aleph_0$-continuous if we can replace the last
 ``then" by ``iff".
\end{definition}

\begin{theorem}
\label{r14} \und{The Canonical Filter Theorem}
Assume {\rm DC} and {\rm AC}$_{\cP(Y)}$.

Assume $\bar \alpha = \langle \alpha_t:t \in Y\rangle \in
{}^Y\text{\rm Ord}$ and $t \in Y \Rightarrow \text{\rm cf}(\alpha_t) 
\ge \theta({\cP}(Y))$ and
$\partial \in \text{\rm ps-pcf}_{\aleph_1\text{\rm -comp}}(\bar \alpha)$
hence is a regular cardinal.  \und{Then} there is $D =
D^{\bar\alpha}_\partial$, an $\aleph_1$-complete filter on $Y$ such that 
$\partial = \text{\rm ps-tcf}(\Pi \bar \alpha/D)$ and $D \subseteq D'$
for any other such $D' \in \text{\rm Fil}^1_{\aleph_1}(D)$.
\end{theorem}

\begin{remark}
\label{r14b}
1) By \ref{r9} there are some such $\partial$.

\noindent
2) We work to use just AC$_{\cP(Y)}$ and not more.

\noindent
3) If $\kappa >\aleph_0$ we can replace ``$\aleph_1$-complete" by
   ``$\kappa$-complete". 
\end{remark}

\begin{PROOF}{\ref{r14}}   Let 
\medskip

\noindent
\begin{enumerate}
\item[$\boxplus_1$]   $(a) \quad \bbD = \{D:D$ is an 
$\aleph_1$-complete filters on $Y$ such
that $(\Pi\bar\alpha/D)$ has

\hskip40pt   pseudo true cofinality $\partial\}$,
\smallskip

\noindent
\item[${{}}$]   $(b) \quad D_* = \cap\{D:D \in \bbD\}$.
\end{enumerate}

Now obviously
\medskip

\noindent
\begin{enumerate}
\item[$(c)$]  $D_*$ is an $\aleph_1$-complete filter on $Y$.
\end{enumerate}

For $A \subseteq Y$ let $\bbD_A = \{D \in \bbD:A \notin D\}$ and
let ${\cP}_* = \{A \subseteq Y:\bbD_A \ne \emptyset\}$.  As
AC$_{{\cP}(Y)}$ we can find $\langle D_A:A \in {\cP}_*\rangle$
such that $D_A \in \bbD_A$ for $A \in {\cP}_*$.  Let $\bbD_* =
\{D_A:A \in {\cP}_*\}$, clearly
\medskip

\noindent
\begin{enumerate}
\item[$\boxplus_2$]  $D_* = \cap\{D:D \in \bbD_*\}$ and $\bbD_*
\subseteq \bbD$ is non-empty.
\end{enumerate}
\medskip

\noindent
As AC$_{\cP_*}$ holds clearly
\medskip

\noindent
\begin{enumerate}
\item[$(*)_0$]   we can 
choose $\langle \bar{\cF}^A:A \in \cP_*\rangle$ such that 
$\bar{\cF}_A$ exemplifies $D_A \in \bbD$ as in 
\ref{r12}(1),(2), 
so in particular is $\aleph_0$-continuous.
\end{enumerate}

For each $\beta < \partial$ let ${\cF}^*_\beta = \cap\{{\cF}^A_\beta:
A \in \cP_*\}$, now
\medskip

\noindent
\begin{enumerate}
\item[$(*)_1$]  ${\cF}^*_\beta \subseteq \Pi \bar\alpha$.
\end{enumerate}

[Why?  As by \ref{r12}(1)(c) we have ${\cF}^A_\beta \subseteq \Pi
\bar\alpha$ for each $A \in \cP_*$.]
\medskip

\noindent
\begin{enumerate}
\item[$(*)_2$]   if $\beta_1 < \beta_2 < \partial,f_1 \in
{\cF}^*_{\beta_1}$ and $f_2 \in {\cF}^*_{\beta_2}$ then $f_1 < f_2$ mod $D_*$.
\end{enumerate}

[Why?  Note that $A \in \cP_* \Rightarrow f_1 <_{D_A} f_2$ by the choice of
$\langle \cF^*_\beta:\beta <\partial\rangle$, hence the set $\{t \in
Y:f_1(t) < f_2(t)\}$ belongs to $D_A$ for every $A \in \cP_*$ hence
by $\boxplus_2$ it belongs to $D_*$ which means that $f_1 <_{D_*} f_2$
as required.]
\medskip

\noindent
\begin{enumerate}
\item[$(*)_3$]  if $f \in \Pi\bar \alpha$ then for some
$\beta_f < \partial$ we have $f' \in \cup\{{\cF}^*_\beta:\beta \in
[\beta_f,\partial)\} \Rightarrow f < f'$ mod $D_*$.
\end{enumerate}

[Why?  For each $A \in \cP_*$ there are $\beta,g$ such that $\beta <
\partial,g \in {\cF}^A_\beta$ and $f < g$ mod $D$ hence $\beta' \in
[\beta +1,\partial) \wedge f' \in {\cF}^A_{\beta'} \Rightarrow f < g < f'$
mod $D_A$.  Let $\beta_A$ be the minimal such ordinal $\beta <
\delta$.  As cf$(\delta) \ge \theta(\cP(Y)) \ge \theta(\cP_*)$,
clearly $\beta_* = \sup\{\beta_A +1:A \in \cP_*\}$ is $< \delta$.  So
$A \in \cP_* \wedge g \in \cup\{\cF^*_\beta:\beta \in
[\beta_*,\delta)) \Rightarrow f <_{D_A} g$.  By $\boxplus_2$ the ordinal
$\alpha_*$ is as required on $\beta_f$.]

Moreover
\medskip

\noindent
\begin{enumerate}
\item[$(*)_4$]   there is a function $f \mapsto \beta_f$ in $(*)_3$.
\end{enumerate}

[Why?  As we can (and will) choose $\beta_f$ as the
minimal $\beta$ such that ...]
\medskip

\noindent
\begin{enumerate}
\item[$(*)_5$]   for every $\beta_* < \partial$ there is
$\beta \in (\beta_*,\partial)$ such that ${\cF}^*_\beta \ne \emptyset$.
\end{enumerate}

[Why?  We choose by induction on $n$, a sequence $\bar \beta_n = \langle
\beta_{n,A}:A \in \cP_*\rangle$ and a sequence $\bar f_n =
\langle f_{n,A}:A \in \cP_*\rangle$ and a function $f_n$ such that
\medskip

\noindent
\begin{enumerate}
\item[$(\alpha)$]   $\beta_n < \partial$ and $m<n \Rightarrow \beta_m <
\beta_n$
\smallskip

\noindent
\item[$(\beta)$]   $\beta_0 = \beta_*$ and for $n > 0$ we let
$\beta_n = \sup\{\beta_{m,A}:m<n,A \in \cP_*\}$
\smallskip

\noindent
\item[$(\gamma)$]  $\beta_{n,A} \in (\beta_n,\partial)$ is minimal such that
there is $f_{n,A} \in {\cF}^A_{\beta_{n,A}}$ satisfying $n=m+1
\Rightarrow f_m < f_{\beta_n,A}$ mod $D_A$
\smallskip

\noindent
\item[$(\delta)$]   $\langle f_{n,A}:A \in \cP_* \rangle$ is a sequence
such that each $f_{n,A}$ are as in clause $(\gamma)$
\smallskip

\noindent
\item[$(\varepsilon)$]  $f_n \in \Pi \bar \alpha$ is defined by $f_n(t)
= \sup\{f_{m,A}(t) +1:A \in \cP_*$ and $m<n\}$.
\end{enumerate}

[Why can we carry the induction?  Arriving to $n$ first,
$f_n$ is well defined $\in \Pi \bar \alpha$ by clause $(\varepsilon)$ as
cf$(\alpha_t) \ge \theta(\cP_*)$ for $t \in Y$.  Second by clause
$(\gamma)$ and the choice of $\big < \langle \bar F^A_\beta:\beta <
\partial \rangle:A \in \cP_* \big>$ in $(*)_0$ the sequence
$\langle \beta_{n,A}:A \in \cP_*\rangle$ is well defined.
Third by clause $(\delta)$ we can choose 
$\langle f_{m,A}:A \in \cP_*\rangle$ because we have {\rm AC}$_{\cP_*}$.
Fourth, $\beta_n$ is well defined by clause $(\beta)$ as cf$(\delta)
\ge \theta(\cP_*)$.

Lastly, the inductive construction is possibly by DC.]
\newline

Let $\beta^* = \cup\{\beta_n:n < \omega\}$ and $f = \sup\langle f_n:n <
\omega\rangle$.  Easily $f \in \cap\{{\cF}^A_{\beta^*}:
A \in \cP_*\}$ as each $\langle \cF^A_\beta:\beta < 
\partial\rangle$ is $\aleph_0$-continuous.]
\medskip

\noindent
\begin{enumerate}
\item[$(*)_6$]   if $f \in \Pi\bar \alpha$ then for some $\beta <
\gamma$ and $f' \in {\cF}^*_\beta$ we have $f < f'$ mod $D^*$.
\end{enumerate}
\medskip

\noindent
[Why?  By $(*)_3 + (*)_5$.]
\newline

So we are done.    
\end{PROOF}
\newpage


\enddocument
\bye